\documentclass[a4paper]{article}

\usepackage{amsmath}
\usepackage{amssymb}
\usepackage{amsthm}
\usepackage{bm}
\usepackage[backend=biber,bibstyle=alphabetic,citestyle=alphabetic,sorting=nyt,maxnames=4]{biblatex}
	\AtEveryBibitem{\clearfield{doi}}
	\AtEveryBibitem{\clearfield{isbn}}
	\AtEveryBibitem{\clearfield{issn}}
	\AtEveryBibitem{\clearfield{url}}
	
	\renewbibmacro{in:}{\ifentrytype{article}{}{\printtext{\bibstring{in}\intitlepunct}}}
\usepackage{calc}
\usepackage{dsfont}
\usepackage{enumitem}
\usepackage[OT2,OT1]{fontenc}
\usepackage[cal=cm,scr=boondoxo]{mathalfa}
\usepackage{pifont}
\usepackage{stmaryrd}
\usepackage{theoremref}
\usepackage{tikz}
	\usetikzlibrary{arrows}
	\usetikzlibrary{patterns}
\usepackage{tikz-cd}
	\tikzset{commutative diagrams/column sep/nixx/.initial=-10pt}

\newcommand\cyr{%
\renewcommand\rmdefault{wncyr}%
\renewcommand\sfdefault{wncyss}%
\renewcommand\encodingdefault{OT2}%
\normalfont
\selectfont}
\DeclareTextFontCommand{\textcyr}{\cyr}

\newcommand{\mc}[1]{{\mathcal{#1}}}			
\newcommand{\ms}[1]{{\mathscr{#1}}}			
\newcommand{\mf}[1]{{\mathfrak{#1}}}			
\newcommand{\bb}[1]{{\mathbb{#1}}}			
\newcommand{\qu}{\overline}				
\newcommand{\mr}{\mathring}				

\DeclareMathOperator{\RE}{Re}\renewcommand{\Re}{\RE}	
\DeclareMathOperator{\tr}{tr}				
\DeclareMathOperator{\dom}{dom}				
\DeclareMathOperator{\ran}{ran}				
\DeclareMathOperator{\ess}{ess}				
\DeclareMathOperator{\supp}{supp}			
\newcommand{\smmatrix}[4]{\Bigl(			
\begin{smallmatrix}
\hspace*{-0.2ex} #1 \hspace*{0.2ex} & \hspace*{0.2ex} #2 \hspace*{-0.2ex} 
\\[0.5ex]
\hspace*{-0.2ex} #3 \hspace*{0.2ex} & \hspace*{0.2ex} #4 \hspace*{-0.2ex}
\end{smallmatrix}
\Bigr)}

\newcommand{\DS}{\colon\mkern3mu}			
\newcommand{\DD}{\mkern4mu}				
\newcommand{\DP}{{.\kern5pt}}				
\newcommand{\DF}{\colon}				
\newcommand{\DE}{\mathrel{\mathop:}=}			
\newcommand{\ED}{=\mathrel{\mathop:}}			

\newcommand{\OI}{{\mf I}}				
\newcommand{\SI}{{\mf S}}				
\newcommand{\LI}[1]{{\mf S}_{[{#1}]}}			
\newcommand{\Orl}[1]{{\mf S}_{\llbracket{#1}\rrbracket}}
\newcommand{\SLI}[1]{{\mf S}_{[{#1}]}^{\bm\circ}}	
\newcommand{\Seq}[1]{{\rm Seq}({#1})}			

\newcommand{\OpA}[1]{A_{[{#1}]}}			
\newcommand{\OpB}[1]{B_{[{#1}]}}			
\newcommand{\OpC}[1]{C_{[{#1}]}}			
\newcommand{\Rot}[2]{
	\circlearrowleft_{{#1}}\mkern-5mu{{#2}}}

\DeclareMathOperator{\Diag}{diag}			

\newcounter{Enum}					
\newenvironment{Enumerate}{\begin{enumerate}[label={\rm({\roman*})}]}{\end{enumerate}}
\newcommand{\Enumref}[1]{{\setcounter{Enum}{#1}{\rm(\roman{Enum})}}}

\newcommand{\descriptionlabelsave}{}			
\newenvironment{Itemize}{%
	\renewcommand{\descriptionlabelsave}{\descriptionlabel}\renewcommand{\descriptionlabel}{$\triangleright$}%
	\begin{description}[leftmargin=15pt,itemindent=-5.2pt]}{%
	\end{description}\renewcommand{\descriptionlabel}{\descriptionlabelsave}}

\newcounter{StepsCount}					
\newenvironment{Steps}{%
	\begin{list}{\ding{\value{StepsCount}}}{\usecounter{StepsCount} \leftmargin=0pt \labelwidth=12pt \itemindent=\labelwidth%
	\itemsep=5pt\listparindent=\parindent} \setcounter{StepsCount}{191}}{\end{list}}

\numberwithin{equation}{section}
\swapnumbers
\theoremstyle{plain}
	\newtheorem{lemma}{Lemma}[section]
	
	\newtheorem{theorem}[lemma]{Theorem}
	\newtheorem{corollary}[lemma]{Corollary}
	\newtheorem{ntheoreM}[lemma]{}
\theoremstyle{definition}
	\newtheorem{definitioN}[lemma]{Definition}
	\newtheorem{ndefinitioN}[lemma]{}
\theoremstyle{remark}
	\newtheorem{remarK}[lemma]{Remark}
	\newtheorem{examplE}[lemma]{Example}
	\newtheorem{nremarK}[lemma]{}
\newcommand{\thlab}[1]{\thlabel{#1}\label{#1 }}
\renewcommand{\qedsymbol}{\raisebox{-2pt}{\large\ding{113}}}
\newcommand{\defendsymbol}{$\lozenge$}
\newcommand{\qedsymbolsave}{\qedsymbol}

\newenvironment{definition}{\begin{definitioN}}{
	\renewcommand{\qedsymbolsave}{\qedsymbol}\renewcommand{\qedsymbol}{\defendsymbol}
	\popQED{\qed}\renewcommand{\qedsymbol}{\qedsymbolsave}\end{definitioN}}

\newenvironment{remark}{\begin{remarK}}{
	\renewcommand{\qedsymbolsave}{\qedsymbol}\renewcommand{\qedsymbol}{\defendsymbol}
	\popQED{\qed}\renewcommand{\qedsymbol}{\qedsymbolsave}\end{remarK}}
\newenvironment{example}{\begin{examplE}}{
	\renewcommand{\qedsymbolsave}{\qedsymbol}\renewcommand{\qedsymbol}{\defendsymbol}
	\popQED{\qed}\renewcommand{\qedsymbol}{\qedsymbolsave}\end{examplE}}

\begin{document}

\begin{flushleft}
	{\Large\bf Canonical systems with discrete spectrum}
	\\[5mm]
	\textsc{
	Roman Romanov
	\,\ $\ast$\,\ 
	Harald Woracek
		\hspace*{-14pt}
		\renewcommand{\thefootnote}{\fnsymbol{footnote}}
		\setcounter{footnote}{2}
		\footnote{The second author was supported the project P\,30715--N35 of the Austrian Science Fund.}
		\renewcommand{\thefootnote}{\arabic{footnote}}
		\setcounter{footnote}{0}
	}
	\\[6mm]
	{\small
	\textbf{Abstract:}
	We study spectral properties of two-dimensional canonical systems $y'(t)=zJH(t)y(t)$, $t\in[a,b)$, 
	where the Hamiltonian $H$ is locally integrable on $[a,b)$, positive semidefinite, and Weyl's limit point case 
	takes place at $b$. We answer the following questions explicitly in terms of $H$:
	\begin{quote}
		{Is the spectrum of the associated selfadjoint operator discrete ?}
		\\[1mm]
		{If it is discrete, what is its asymptotic distribution ?} 
	\end{quote}
	Here asymptotic distribution means summability and limit superior conditions relative to 
	comparison functions growing sufficiently fast. 
	Making an analogy with complex analysis, this correponds to convergence class and type w.r.t.\ proximate 
	orders having order larger than $1$. 
	It is a surprising fact that these properties depend only on the diagonal entries of $H$. 

	In 1968 L.de~Branges posed the following question as a fundamental problem:
	\begin{quote}
		{Which Hamiltonians are the structure Hamiltonian of some\\ de~Branges space ?}
	\end{quote}
	We give a complete and explicit answer.
	\\[3mm]
	}
\end{flushleft}


%
%
%
\section{Introduction}

We study the spectrum of the selfadjoint model operator associated with a two-dimensional canonical system 
\begin{equation}\label{Q101}
	y'(t)=zJH(t)y(t),\quad t\in[a,b)
	.
\end{equation}
Here $H$ is the Hamiltonian of the system, $-\infty<a<b\leq\infty$, $J$ is the symplectic matrix $J\DE\smmatrix 0{-1}10$, 
and $z\in\bb C$ is the eigenvalue parameter. We assume throughout that $H$ satisfies 
\begin{Itemize}
\item $H\in L^1\big([a,c),\bb R^{2\times 2}\big)$, $c\in(a,b)$, and $\{t\in[a,b)\DS H(t)=0\}$ has measure $0$,
\item $H(t)\geq 0$, $t\in[a,b)$ a.e.\ and $\int_a^b\tr H(t)\DD dt=\infty$. 
\end{Itemize}
Differential equations of this form orginate from Hamiltonian mechanics, and appear frequently in theory and applications. 
Various kinds of equations can be rewritten to the form \eqref{Q101}, and several problems of classical analysis 
can be treated with the help of canonical systems. For example we mention 
Schr\"odinger operators \cite{remling:2002}, 
Dirac systems \cite{sakhnovich:2002},
or the extrapolation problem of stationary Gaussian processes via Bochners theorem \cite{krein.langer:2014}. 
Other instances can be found, e.g., in \cite{kac:1999,kac:1999a}, \cite{kaltenbaeck.winkler.woracek:bimmel}, 
\cite{akhiezer:1961}, or \cite{arov.dym:2008}. 

The direct and inverse spectral theory of the equation \eqref{Q101} was developed in \cite{gohberg.krein:1967,debranges:1968}. 
Recent references are \cite{remling:2018,romanov:1408.6022v1}.

With a Hamiltonian $H$ a Hilbert space $L^2(H)$ is associated, and in $L^2(H)$ a selfadjoint operator $\OpA{H}$
is given by the differential expression \eqref{Q101} and by prescribing the boundary condition $(1,0)y(a)=0$ 
(in one exceptional situation $\OpA{H}$ is a multivalued operator, but this is only a technical difficulty). 
This operator model behind \eqref{Q101} was given its final form in \cite{kac:1983,kac:1984}. A more accessible reference 
is \cite{hassi.snoo.winkler:2000}, and the relation with de~Branges' work on Hilbert spaces of entire functions 
was made explicit in \cite{winkler:1993,winkler:1995}. 

In the present paper we answer the following questions:
\begin{quote}
	\textit{Is the spectrum of $\OpA{H}$ discrete ?}
	\\[1mm]
	\textit{If $\sigma(\OpA{H})$ is discrete, what is its asymptotic distribution ?} 
\end{quote}
The question about asymptotic distribution is understood as the problem of finding the convergence exponent 
and the upper density of eigenvalues in terms of the Hamiltonian. 

\subsubsection*{Discreteness of the spectrum}

In our first theorem we characterise discreteness of $\sigma(\OpA{H})$. 

\begin{theorem}\thlab{Q102}
	Let $H=\smmatrix{h_1}{h_3}{h_3}{h_2}$ be a Hamiltonian on $[a,b)$ and assume that $\int_a^b h_1(s)\DD ds<\infty$. 
	Then $\sigma(\OpA{H})$ is discrete if and only if 
	\begin{equation}\label{Q130}
		\lim_{t\nearrow b}\Big(\int_t^b h_1(s)\DD ds\cdot\int_a^t h_2(s)\DD ds\Big)=0
		.
	\end{equation}
\end{theorem}

\begin{remark}\thlab{Q103}
	The assumption that $\int_a^b h_1(s)\DD ds<\infty$ in \thref{Q102} is made for normalisation and is 
	no loss in generality. 
	First, a necessary condition that $0\notin\sigma_{\ess}(\OpA{H})$ is that there exists some angle $\phi\in\bb R$ 
	such that $\int_a^b(\cos\phi,\sin\phi)H(s)(\cos\phi,\sin\phi)^*\DD ds<\infty$. Second, applying 
	rotation isomorphism always allows to reduce to the case that $\phi=0$. We will give details in Section~5.2.
\end{remark}

\noindent
Let us remark that \thref{Q102} yields a new proof of the discreteness criterion for strings given by I.S.Kac and M.G.Krein 
in \cite[Theorema~4,5]{kac.krein:1958}, of \cite[Theorem~1.4]{remling.scarbrough:1811.07067v1}, and of 
\cite[Theorem~1]{kac:1995}.

\subsubsection*{Structure Hamiltonians of de~Branges spaces}

Recall that a de~Branges space $\mc H(E)$ is a reproducing kernel Hilbert space of entire functions with 
certain additional properties, whose kernel is generated by a Hermite-Biehler function $E$. For each de~Branges space there exists 
a unique maximal chain of de~Branges subspaces $\mc H(E_t)$, $t\leq 0$, contained isometrically (on exceptional intervals only 
contractively) in $\mc H(E)$. The generating Hermite-Biehler functions $E_t$ satisfy a canonical system on the interval $(-\infty,0]$ 
with some Hamiltonian $H$, and this Hamiltonian is called the structure Hamiltonian of $\mc H(E)$. 

L.de~Branges identified in \cite[Theorem~IV]{debranges:1961} a particular class of Hamiltonians which are 
structure Hamiltonians of de~Branges spaces. 
A mild generalisation of de~Branges' theorem can be found in \cite[Theorem~4.11]{linghu:2015}, 
and a further class of structure Hamiltonians is identified (in a different language) by the already mentioned work of 
I.S.Kac and M.G.Krein \cite{kac.krein:1958} and its mild generalisation \cite{remling.scarbrough:1811.07067v1}.
These classes do by far not exhaust the set of all structure Hamiltonians.
In 1968, after having finalised his theory of Hilbert spaces of entire functions, de~Branges posed the following question 
as a fundamental problem, cf.\ \cite[p.140]{debranges:1968}:
\begin{quote}
	\textit{Which Hamiltonians $H$ are the structure Hamiltonian of some\\ de~Branges space $\mathcal H(E)$ ?}
\end{quote}
In the following decades there was no significant progress towards a solution of this question. 
One result was claimed by I.S.Kac in 1995; proofs have never been published. 
He states a sufficient condition and a (different) necessary condition for $H$ to be a structure Hamiltonian. 
Unfortunately, his conditions are difficult to handle.

The connection with \thref{Q102} is the following: 
a Hamiltonian is the structure Hamiltonian of some de~Branges space $\mc H(E)$, if and only if 
the operator $\OpA{\tilde H}$ associated with the reversed Hamiltonian 
\[
	\tilde H(t)\DE\smmatrix 100{-1}H(-t)\smmatrix 100{-1},\quad t\in[0,\infty)
	,
\]
has discrete spectrum. This can be seen by a simple ``juggling with fundamental solutions''--argument. 
A proof based on a different argument was published in \cite{kac:2007}, see also \cite[Theorem~2.3]{linghu:2015}. 

Hence we obtain from \thref{Q102} a complete and explicit answer to de~Branges' question. 

\subsubsection*{Summability properties}

We turn to discussing the asymptotic distribution of $\sigma(\OpA{H})$. 
Consider a Hamiltonian $H$ with discrete spectrum. Then its spectrum is a (finite or infinite) sequence of simple eigenvalues 
without finite accumulation point. 
If $\sigma(\OpA{H})$ is finite, any questions about the asymptotic behaviour of the eigenvalues are obsolete. 
Moreover, under the normalisation that $\int_a^b h_1(s)\DD ds<\infty$, the point $0$ is not an eigenvalue of $\OpA{H}$. 
Hence, we can think of $\sigma(\OpA{H})$ as a sequence $(\lambda_n)_{n=1}^\infty$ of pairwise different real numbers 
arranged such that 
\begin{equation}\label{Q131}
	0<|\lambda_1|\leq|\lambda_2|\leq|\lambda_3|\leq\ldots 
\end{equation}
In our second theorem we characterise summability of the sequence $(\lambda_n^{-1})_{n=1}^\infty$ relative to 
suitable comparison functions. 
In particular, this answers the question whether $(\lambda_n^{-1})_{n=1}^\infty\in\ell^p$ when $p>1$. 
The only known result in this direction is \cite[Theorem~2.4]{kaltenbaeck.woracek:hskansys}, which settles the case $p=2$; 
we reobtain this theorem. 

As comparison functions we use functions $\ms g$ defined on the ray $[0,\infty)$ and taking values in $(0,\infty)$
which satisfy:
\begin{Itemize}
\item The limit $\rho_{\ms g}\DE\lim_{r\to\infty}\frac{\log\ms g(r)}{\log r}$ exists and belongs to $(0,\infty)$.
\item The function $\ms g$ is continuously differentiable with $\ms g'(r)>0$, 
	and $\lim_{r\to\infty}\frac{r\ms g'(r)}{\ms g(r)}=\rho_\ms g$.
\end{Itemize}
Functions of this kind are known as growth functions; the number $\rho_{\ms g}$ is called the order of $\ms g$.
They form a comparison scale which is finer than the scale of powers $r^\rho$. 
The history of working with growth scales other than powers probably starts with the paper \cite{lindeloef:1905}, where 
E.Lindel\"of compared the growth of an entire function with functions of the form 
\[
	\ms g(r)=r^\rho\cdot\big(\log r\big)^{\beta_1}\cdot\big(\log\log r\big)^{\beta_2}\cdot\ldots\cdot
	\big(\underbrace{\log\cdots\log}_{\text{\footnotesize$m$\textsuperscript{th}-iterate}}r\big)^{\beta_m}
	\quad\text{for $r$ large}
	.
\]
In what follows the reader may think of $\ms g(r)$ for simplicity as a concrete function of this form, or simply as a power $r^\rho$.

\begin{theorem}\thlab{Q104}
	Let $H=\smmatrix{h_1}{h_3}{h_3}{h_2}$ be a Hamiltonian on an interval $[a,b)$ such that $\int_a^b h_1(s)\DD ds<\infty$
	and that $\OpA{H}$ has discrete spectrum. 
	Moreover, assume that $h_1$ does not vanish a.e.\ on any interval $(c,b)$ with $c\in(a,b)$. 
	Let $\ms g$ be a growth function with order $\rho_\ms g>1$. Then 
	\begin{multline*}
		\sum_{n=1}^\infty\frac 1{\ms g(|\lambda_n|)}<\infty\ \Leftrightarrow\quad
		\\
		\int_a^b\bigg[
		\ms g\bigg(\Big(\int_t^b h_1(s)\DD ds\cdot\int_a^t h_2(s)\DD ds\Big)^{-\frac 12}\bigg)
		\bigg]^{-1}
		\mkern-5mu\cdot\mkern1mu
		\frac{h_1(t)\DD dt}{\int_t^b h_1(s)\DD ds}<\infty
		.
	\end{multline*}
\end{theorem}

\begin{remark}\thlab{Q105}
	For the same reasons as explained in \thref{Q103}, the assumption that $\int_a^b h_1(s)\DD ds<\infty$ is 
	just a normalisation and no loss in generality. 
	Also the assumption that $h_1$ cannot vanish a.e.\ on any interval $(c,b)$ is no loss of generality. 
	The reason being that, if $h_1$ does vanish on an interval of this form, then the Krein-de~Branges 
	formula, cf.\ \cite[p.369 (english translation)]{krein:1951}, \cite[Theorem~X]{debranges:1961}, says that 
	\begin{equation}\label{Q108}
		\lim_{n\to\infty}\frac{\lambda_n^+}n=\lim_{n\to\infty}\frac{\lambda_n^-}n=
		\pi\int_a^b\sqrt{\det{H(s)}}\DD ds
		,
	\end{equation}
	where $\lambda_n^\pm$ denote the sequences of positive and negative, respectively, eigenvalues arranged according to 
	increasing modulus. In particular, the series $\sum_{n=1}^\infty\frac 1{\ms g(|\lambda_n|)}$ converges whenever $\rho_\ms g>1$. 
\end{remark}

\noindent
Let us note that, besides the condition for square summability given in \cite{kaltenbaeck.woracek:hskansys}, 
\thref{Q104} also yields new proofs of the results on the genus of the spectrum of a string given in 
\cite[p.139f]{kac.krein:1958} and \cite[Theorema~1,2]{kac:1962}, and of \cite{kac:1986} for orders between $\frac 12$ and $1$. 

\subsubsection*{Limit superior properties}

In our third theorem, we characterise $\limsup$--properties of the sequence $(\lambda_n)_{n=1}^\infty$, again relative to 
growth functions $\ms g$ with $\rho_\ms g>1$. While the characterisations in \thref{Q102,Q104} are perfectly explicit in terms 
of $H$, the conditions occurring in this context are somewhat more complicated. The reason for this is 
intrinsic, and manifests itself in the necessity to pass to the nonincreasing rearrangement of a certain sequence. 

\begin{theorem}\thlab{Q106}
	Let $H=\smmatrix{h_1}{h_3}{h_3}{h_2}$ be a Hamiltonian on an interval $[a,b)$ such that $\int_a^b h_1(s)\DD ds<\infty$
	and that $\OpA{H}$ has discrete spectrum. 
	Moreover, assume that $h_1$ does not vanish a.e.\ on any interval $(c,b)$ with $c\in(a,b)$. 
	Let $\ms g$ be a growth function with order $\rho_\ms g>1$. 

	Choose a right inverse $\chi$ of the nonincreasing surjection
	\[
		\left\{
		\begin{array}{rcl}
			[a,b] & \to & [0,1]
			\\
			t & \mapsto & \Big(\int_a^bh_1(s)\DD ds\Big)^{-1}\Big(\int_t^bh_1(s)\DD ds\Big)
		\end{array}
		\right.
	\]
	and let $(\omega_n^*)_{n\in\bb N}$ be the nonincreasing rearrangement of the sequence 
	$(\omega_n)_{n\in\bb N}$ defined as 
	\[
		\omega_n\DE 2^{-\frac n2}\bigg(\int\limits_{\chi(2^{1-n})}^{\chi(2^{-n})}h_2(s)\DD ds\bigg)^{\frac 12}
		,\quad n\in\bb N
		.
	\]
	Then
	\begin{Enumerate}
	\item ${\displaystyle
		\limsup_{n\to\infty}\frac{n}{\ms g(|\lambda_n|)}<\infty
		\ \Leftrightarrow\ 
		\limsup_{n\to\infty}\frac n{\ms g((\omega_n^*)^{-1})}<\infty
		}$,
	\item ${\displaystyle
		\lim_{n\to\infty}\frac{n}{\ms g(|\lambda_n|)}=0
		\ \Leftrightarrow\ 
		\lim_{n\to\infty}\frac n{\ms g((\omega_n^*)^{-1})}=0
		}$.
	\end{Enumerate}
\end{theorem}

\noindent
Remember here \thref{Q105}.

\subsubsection*{Outline of the proofs}

The proof of \thref{Q102,Q104,Q106} proceeds through four stages. 
\begin{Steps}
\item The first stage is to pass from eigenvalue distribution to operator theoretic properties. 
	This is done in a standard way using symmetrically normed operator ideals: 
	discreteness of $\sigma(\OpA{H})$ is equivalent to $(\OpA{H}-z)^{-1}$ being compact, 
	summability properties of $\sigma(\OpA{H})$ are equivalent to $(\OpA{H}-z)^{-1}$ belonging to Orlicz ideals, 
	and $\limsup$--properties of $\sigma(\OpA{H})$ are equivalent to $(\OpA{H}-z)^{-1}$ belonging to Lorentz spaces.
\item The reader has certainly observed the -- probably surprising -- fact that the conditions in our theorems do not 
	involve the off-diagonal entry $h_3$ of the Hamiltonian $H$. The second stage is to prove an Independence Theorem which 
	says that membership of resolvents $(\OpA{H}-z)^{-1}$ in operator ideals $\OI$ is indeed independent of $h_3$, provided 
	$\OI$ possesses a certain additional property. This additional property is that a weak variant of Matsaev's Theorem 
	on real parts of Volterra operators holds in $\OI$. 
\item In a work of A.B.Aleksandrov, S.Janson, V.V.Peller, and R.Rochberg, membership in Schatten classes 
	of integral operators whose kernel has a particular form is characterised using a dyadic discretisation method. 
	The third stage is to realise that a minor generalisation of one of their results suffices to prove the mentioned 
	weak Matsaev Theorem in the ideal $\SI_\infty$ of all compact operators. 
	For Orlicz- and Lorentz ideals, it is known that (the full) Matsaev Theorem holds.
	Thus the Independence Theorem stated in \ding{193} will apply to all ideals occurring in \ding{192}.
\item The final stage is to characterise membership in the mentioned ideals for a diagonal Hamiltonian (meaning that $h_3=0$).
	This again rests on the discretisation method from \cite{aleksandrov.janson.peller.rochberg:2002}, which 
	yields characterisations of a sequential form (as the one stated in \thref{Q106}) for all ideals occurring in \ding{192}.
	For the cases of $\SI_\infty$ and Orlicz ideals, sequential characterisations can be rewritten to a continuous form 
	(as stated in \thref{Q102,Q104}). This is nearly obvious for $\SI_\infty$, 
	while for Orlicz ideals a little more effort and passing to dual spaces is needed. 
\end{Steps}

\subsubsection*{The threshold $\bm{\rho_\ms g=1}$}

The Krein-de~Branges formula \eqref{Q108} implies that for every Hamiltonian $H$ whose determinant 
does not vanish a.e., the spectrum $\sigma(\OpA{H})$, if discrete, satisfies 
$\liminf_{n\to\infty}\frac n{|\lambda_n|}>0$. On the other hand, $\sigma(\OpA{H})$ can be arbitrarily sparse if $\det H=0$ a.e. 
It is not difficult to find Hamiltonians $H=\smmatrix{h_1}{h_3}{h_3}{h_2}$ whose spectrum 
is discrete and satisfies $\lim_{n\to\infty}\frac n{|\lambda_n|^\rho}\in(0,\infty)$ for some $\rho<1$, 
but in the same time $h_1h_2$ does not vanish a.e., see \thref{Q132}. 
For each such Hamiltonian and every Schatten-von Neumann ideal $\SI_p$ with $\rho<p\leq 1$, 
the Independence Theorem mentioned in \ding{193} fails. 
This shows that our method necessarily must break down at (and below) trace class, i.e., growth of speed $\ms g(r)\DE r$. 

On a less concrete level, growth of order $1$ is a threshold because of (at least) four reasons. 
\begin{Itemize}
\item Orders larger than $1$, meaning eigenvalue distriubution more dense than integers, can occur only from 
	the behaviour of tails of $H$ at its singular endpoint $b$. 
	In fact, for $\rho>1$, the spectrum of $\OpA{H}$ is discrete with convergence exponent $\rho$ if and only if 
	for some $c\in(a,b)$ the spectrum of $\OpA{H|_{[c,b)}}$ is discrete with convergence exponent $\rho$. 
	
	Contrasting this, orders less than $1$ will in general accumulate over the whole interval $(a,b)$.
	In fact, it may happen that $\sigma(\OpA{H})$ has convergence exponent $1$ while for every $c\in(a,b)$ 
	the spectrum of the tail $H|_{[c,b)}$ has convergence exponent $0$. 
\item Entire functions of bounded type have very specific properties related to exponential type. 
	In complex analysis, orders larger than $1$ are usually considered as more stable than smaller orders. 
\item The theory of symmetrically normed operator ideals is significantly more complicated 
	for ideals close to trace class than for ideals containing some Schatten-von Neumann ideal $\SI_p$ with $p>1$. 
	When going even below trace class, a lot of the theory breaks down completely. 
\item Rewriting asymptotic conditions on the spectral distribution to conditions on membership in Orlicz- and Lorentz ideals 
	is not anymore possible when coming close to trace class. 
\end{Itemize}
Let us now give two examples which illustrate our results. They are simple, and given by Hamiltonians related to a string, 
but, as we hope, still illustrative. At this point we only state their spectral properties; the proof is given in Section~5.3. 

\begin{example}\thlab{Q109}
	Given $\alpha>1$ and $\alpha_1,\alpha_2\in\bb R$, we consider the Hamiltonian (to avoid bulky notation, we skip 
	indices $\alpha,\alpha_1,\alpha_2$ at $h_2$)
	\[
		H_{\alpha;\alpha_1,\alpha_2}(t)\DE
		\begin{pmatrix}
			1 & 0
			\\
			0 & h_2(t)
		\end{pmatrix}
		,\quad t\in[0,1)
		,
	\]
	where 
	\begin{equation}\label{Q133}
		h_2(t)\DE\Big(\frac 1{1-t}\Big)^\alpha\Big(1+\log\frac 1{1-t}\Big)^{-\alpha_1}
		\Big(1+\log^+\log\frac 1{1-t}\Big)^{-\alpha_2},\quad t\in[0,1)
		.
	\end{equation}
	Since $\alpha>1$, we have $\int_0^1 h_2(t)\DD dt=\infty$. 

	If $\alpha>2$, then $0$ belongs to the essential spectrum of $\OpA{H_{\alpha;\alpha_1,\alpha_2}}$, and if 
	$\alpha\in(1,2)$, then the spectrum is discrete with convergence exponent $1$ but 
	$\liminf_{n\to\infty}\frac n{|\lambda_n|}>0$, in particular, 
	$\sum_{n=1}^\infty\frac 1{|\lambda_n|}=\infty$. 

	A behaviour between those extreme situations occurs when $\alpha=2$. First, the spectrum of 
	$\OpA{H_{2;\alpha_1,\alpha_2}}$ is discrete, if and only if 
	\[
		(\alpha_1>0)\vee(\alpha_1=0,\alpha_2>0)
		.
	\]
	For such parameter values, the convergence exponent of the spectrum is
	\begin{equation}\label{Q134}
		\text{conv.exp.\ of }\sigma(\OpA{H_{2;\alpha_1,\alpha_2}})=
		\begin{cases}
			\infty &\hspace*{-3mm},\quad \alpha_1=0
			,
			\\
			\frac 2{\alpha_1} &\hspace*{-3mm},\quad \alpha_1\in(0,2)
			,
			\\
			1 &\hspace*{-3mm},\quad \alpha_1\geq 2
			.
		\end{cases}
	\end{equation}
	For $\alpha_1\in(0,2)$, we have a more refined $\limsup$-property relative to a comparison function which is not 
	a power:
	\[
		0<\limsup_{n\to\infty}\frac n{|\lambda_n|^{\frac 2{\alpha_1}}(\log|\lambda_n|)^{-\frac{\alpha_2}{\alpha_1}}}
		<\infty
		.
	\]
\end{example}

\begin{example}\thlab{Q132}
	Given $\alpha_1>0$ and $\alpha_2\in\bb R$ consider the Hamiltonian (again indices at $h_2$ are skipped)
	\[
		\mr H_{\alpha_1,\alpha_2}\DE
		\begin{pmatrix}
			1 & -\sqrt{h_2(t)}
			\\
			-\sqrt{h_2(t)} & h_2(t)
		\end{pmatrix}
		,\quad t\in[0,1)
		,
	\]
	where $h_2$ is as in \eqref{Q133} with $\alpha=2$. 
	Then $\sigma(\OpA{\mr H_{\alpha_1,\alpha_2}})$ is discrete, and its convergence exponent is
	\begin{equation}\label{Q135}
		\text{conv.exp.\ of }\sigma(\OpA{\mr H_{\alpha_1,\alpha_2}})=
		\begin{cases}
			\frac 2{\alpha_1} &\hspace*{-3mm},\quad \alpha_1\in(0,4)
			,
			\\
			\frac 12 &\hspace*{-3mm},\quad \alpha_1\geq 4
			.
		\end{cases}
	\end{equation}
	The diagonalisation of $\mr H_{\alpha_1,\alpha_2}$, i.e., the Hamiltonian obtained by skipping its off-diagonal entries, 
	is $H_{2;\alpha_1,\alpha_2}$. Comparing the convergence exponents computed in \eqref{Q134} and \eqref{Q135}, 
	illustrates validity of the Independence Theorem 
	from \ding{193} as long as the convergence exponent is not less than $1$, and its failure for other values. 
\end{example}

\subsubsection*{Organisation of the manuscript}

The structuring of this article is straightforward. 
We start off in Section~2 with proving the central Independence Theorem mentioned in \ding{193}. 
Section~3 contains the proof of \thref{Q102}, and Section~4 the proofs of \thref{Q104,Q106}. 
Finally, in Section~5, we give a fourth theorem in the spirit of the above three theorems, 
disuss the normalisation condition $\int_a^bh_1(s)\DD ds<\infty$, and provide details for the \thref{Q109,Q132}.

As mentioned in \ding{194} and \ding{195} above, our arguments use a (very) minor generalisation of the AJPR-results. 
This is established in just the same way as the results of 
\cite{aleksandrov.janson.peller.rochberg:2002} with only a few technical additions.
For the convenience of the reader we provide full details in Appendix~A.
Moreover, in Appendix~B, we provide detailed proof for some elementary facts being used in the text, and in 
Appendix~C we make the connection of our \thref{Q102} with \cite[Theorem~1]{kac:1995}.

\section{The Independence Theorem}

Let $\mc H$ be a Hilbert space and $\mc B(\mc H)$ the set of all bounded linear operators on $\mc H$. 
For an operator $T\in\mc B(\mc H)$ we denote by $a_n(T)$ the $n$-th approximation number of $T$, i.e.,  
\[
	a_n(T)\DE\inf\big\{\|T-A\|\DS A\in\mc B(\mc H),\dim\ran A<n\big\}
	,\quad n\in\bb N
	.
\]
The Calkin correspondence \cite{calkin:1941} is the map assigning to each $T\in\mc B(\mc H)$ 
the sequence $(a_n(T))_{n=1}^\infty$ of its approximation numbers. 

An operator ideal $\OI$ in $\mc H$ is a two-sided ideal of the algebra $\mc B(\mc H)$. 
Every proper operator ideal $\OI$ contains the ideal of all finite rank operators, and, provided $\mc H$ is separable, 
is contained in the ideal $\SI_\infty$ of all compact operators. 
Moreover, every operator ideal contains with an operator $T$ also its adjoint $T^*$.

Via the Calkin correspondence, operator ideals can be identified with certain sequence spaces. 
Recall \cite[Theorem~1]{garling:1967}: there is a bijection {\rm Seq} of the set of all operator ideals of $\mc H$ onto the 
set of all solid symmetric sequence spaces\footnote{%
	A linear subspace $\mc S$ of $\bb C^{\bb N}$ is called solid, if 
	\[
		(\alpha_n)_{n=1}^\infty\in\mc S\wedge|\beta_n|\leq|\alpha_n|,n\in\bb N
		\ \Rightarrow\ 
		(\beta_n)_{n=1}^\infty\in\mc S
		,
	\]
	and it is called symmetric, if 
	\[
		(\alpha_n)_{n=1}^\infty\in\mc S,\sigma\text{ permutation of }\bb N
		\ \Rightarrow\ 
		(\alpha_{\sigma(n)})_{n=1}^\infty\in\mc S
		.
	\]
}, 
such that for all $T\in\mc B(\mc H)$ 
\[
	T\in\OI\quad\Leftrightarrow\quad(a_n(T))_{n=1}^\infty\in\Seq{\OI}
	.
\]
For example, the ideal $\SI_\infty$ of all compact operators corresponds to $c_0$, the trivial ideal 
$\mc B(H)$ to $\ell^\infty$, and the Schatten--von Neumann classes $\SI_p$ to $\ell^p$. 
Taking the viewpoint of sequence spaces is natural in (at least) two respects. 
\begin{Itemize}
\item It allows to compare ideals in $\mc B(\mc H)$ for different base spaces $\mc H$.
	A solid symmetric sequence space $\mc S$ invokes the family of ``same-sized'' operator ideals
	\[
		\big\{T\in\mc B(\mc H)\DS (a_n(T))_{n=1}^\infty\in\mc S\big\},\quad\mc H\text{ Hilbert space}
		.
	\]
\item Virtually all examples of operator ideals $\OI$ which ``appear in nature'' are defined by a specifying 
	their sequence space $\Seq{\OI}$. 
\end{Itemize}
From now on we do not anymore distinguish between sequence spaces and operator ideals, and always speak of an operator ideal.

A central role is played by integral operators whose kernel has a very special form. 
Let $-\infty\leq a<b\leq\infty$, and $\kappa,\varphi\DF(a,b)\to\bb C$ be measurable functions 
such that $\kappa\in L^2(a,b)$ and $\mathds{1}_{(a,c)}\varphi\in L^2(a,b)$ for every $c\in(a,b)$. 
Then we consider the (closed, but possibly unbounded) integral operator $T$ in $L^2(a,b)$ with kernel
\begin{equation}\label{Q110}
	\parbox{\textwidth-15mm}{${\displaystyle\mkern50mu
	T:\quad \mathds{1}_{t<s}(t,s)\varphi(t)\qu{\kappa(s)}
	}$}
\end{equation}
Explicitly, this is the operator acting as 
\[
	(Tf)(t)\DE \varphi(t)\int_t^bf(s)\qu{\kappa(s)}\DD ds,\quad t\in(a,b)
	,
\]
on its natural maximal domain
\[
	\dom T\DE \Big\{f\in L^2(a,b)\DS 
	\Big(t\mapsto\varphi(t)\int_t^bf(s)\qu{\kappa(s)}\DD ds\Big)\in L^2(a,b)\Big\}
	.
\]
In the next definition we single out a crucial property which an operator ideal $\OI$ may or may not have.

\begin{definition}\thlab{Q111}
	Let $\OI$ be an operator ideal. We say that the weak Matsaev Theorem holds in $\OI$, if the following
	statement is true.
	\begin{Itemize}
	\item Let $-\infty\leq a<b\leq\infty$, let $\kappa,\varphi\DF(a,b)\to\bb C$ be measurable functions 
		such that $\kappa\in L^2(a,b)$ and $\mathds{1}_{(a,c)}\varphi\in L^2(a,b)$ for every $c\in(a,b)$, 
		and let $T$ be the integral operator with kernel \eqref{Q110}. 
		Then $\Re T\in\OI$ implies $T\in\OI$. 
	\end{Itemize}
\end{definition}

\noindent
We are going to compare a Hamiltonian $H$ with its diagonal part.

\begin{definition}\thlab{Q112}
	Let $H$ be a Hamiltonian and write $H=\smmatrix{h_1}{h_3}{h_3}{h_2}$. Then we denote the corresponding 
	diagonal Hamiltonian as 
	\[
		\Diag H\DE
		\begin{pmatrix}
			h_1 & 0
			\\
			0 & h_2
		\end{pmatrix}
		.
		\vspace*{-3mm}
	\]
\end{definition}

\begin{theorem}[Independence Theorem]\thlab{Q113}
	Let $H$ be a Hamiltonian defined on an interval $[a,b)$, and assume that $\int_a^b h_1(s)\DD ds<\infty$.
	Moreover, let $\OI$ be an operator ideal. 

	Then $\OpA{\Diag H}\in\OI$ implies that $\OpA{H}\in\OI$. If the weak Matsaev Theorem holds in $\OI$, 
	also the converse holds. 
\end{theorem}

\noindent
Before giving the proof, we have to recall some known facts about the model operator $\OpA{H}$. 
The first lemma is folklore; one possible reference is \cite{kaltenbaeck.woracek:hskansys} where it appears implicitly.

\begin{lemma}\thlab{Q114}
	Under the assumption that $\int_a^b h_1(s)\DD ds<\infty$, the operator $\OpA{H}$ is injective and 
	its inverse $\OpB{H}\DE\OpA{H}^{-1}$ acts as
	\begin{equation}\label{Q125}
		(\OpB{H}f)(t)=-\lim_{c\nearrow b}\int_a^c
		\begin{pmatrix}
			0 & \mathds{1}_{s<t}(t,s)
			\\
			\mathds{1}_{s>t}(t,s) & 0
		\end{pmatrix}
		H(s)f(s)\DD ds
		,
	\end{equation}
	on the domain 
	\begin{equation}\label{Q126}
		\dom\OpB{H}=\bigg\{f\in L^2(H)\DS
		\parbox{50mm}{\small $\lim_{c\nearrow b}(0,1)\int_a^c JH(s)f(s)\DD ds$ exists,\\[1mm]
		r.h.s.\ of {\rm(\ref{Q125})} belongs to $L^2(H)$}
		\bigg\}
		.
	\end{equation}
	\popQED\qed
\end{lemma}

\noindent
Denote by $L^2(Idt)$ the $L^2$-space of $2$-vector valued functions on $(a,b)$ 
with respect to the matrix measure $\smmatrix 1001\DD dt$. 
The function 
\[
	\Phi\DF f(t)\mapsto H(t)^{\frac 12}f(t)
\] 
maps the model space $L^2(H)$ isometrically onto some closed subspace of $L^2(Idt)$. 
This holds since, by its definition, $L^2(H)$ is a closed subspace of the $L^2$-space of $2$-vector valued 
functions on $(a,b)$ with respect to the matrix measure $H(t)dt$.

Let $\OpC{H}$ be the (closed, but possibly unbounded) integral operator on $L^2(Idt)$ with kernel 
\begin{equation}\label{Q115}
	\parbox{\textwidth-15mm}{${\displaystyle\mkern50mu
	\OpC{H}:\quad
	-H(t)^{\frac 12}
	\begin{pmatrix}
		0 & \mathds{1}_{s<t}(t,s)
		\\
		\mathds{1}_{s>t}(t,s) & 0
	\end{pmatrix}
	H(s)^{\frac 12}
	}$}
\end{equation}
and the natural maximal domain. 

The next lemma says that the operator $\OpB{H}$ can be transformed into $\OpC{H}$,
and was shown in \cite[Proof of Lemma~2.2]{kaltenbaeck.woracek:hskansys}.

\begin{lemma}\thlab{Q116}
	Assume that $\int_a^b h_1(s)\DD ds<\infty$, and denote by $P$ the orthogonal projection of $L^2(Idt)$ 
	onto $\ran\Phi$. Then 
	\[
		\OpB{H}=\Phi^{-1}P\OpC{H}\Phi\quad\text{and}\quad \OpC{H}=\Phi\OpB{H}\Phi^{-1}P
		.
	\]
	\popQED\qed
\end{lemma}

\noindent
As a consequence of \thref{Q116}, the operators $\OpB{H}$ and $\OpC{H}$ are together bounded or unbounded, 
and if they are bounded their approximation numbers coincide. Thus, for every operator ideal $\OI$, we have 
\[
	\OpB{H}\in\OI\ \Leftrightarrow\ \OpC{H}\in\OI
	.
\]
The following simple computation is a key step to the proof of \thref{Q113}.

\begin{lemma}\thlab{Q117}
	Let $H$ be a Hamiltonian on $[a,b)$ with $\int_a^b h_1(s)\DD ds<\infty$. 
	Denote 
	\[
		H(t)^{\frac 12}=
		\begin{pmatrix}
			v_1(t) & v_3(t)
			\\
			v_3(t) & v_2(t)
		\end{pmatrix}
		,\quad t\in [a,b)
		,
	\]
	and let $T_{ij}$, $(i,j)\in\{2,3\}\times\{1,3\}$, be the integral operators in $L^2(a,b)$ with kernel 
	\[
		\parbox{\textwidth}{${\displaystyle\mkern50mu
		T_{ij}:\quad \mathds{1}_{t<s}(t,s)v_i(t)v_j(s)
		}$}
	\]
	Then 
	\[
		\OpC{H}f=-
		\begin{pmatrix}
			T_{31}+T_{31}^* & T_{21}^*+T_{33}
			\\[2mm]
			T_{21}+T_{33}^* & T_{23}+T_{23}^*
		\end{pmatrix}
		f,\quad f\in\dom\OpC{H}
		.
	\]
\end{lemma}
\begin{proof}
	Multiplying out the kernel \eqref{Q115} of the integral operator $\OpC{H}$ gives
	\[
		\begin{pmatrix}
			\rule{7pt}{0pt}\parbox{50mm}{%
			$\mathds{1}_{t<s}(t,s)v_3(t)v_1(s)$\\ \rule{15mm}{0pt}$+\mathds{1}_{t>s}(t,s)v_1(t)v_3(s)
			$} &
			\parbox{49mm}{%
			$\mathds{1}_{t<s}(t,s)v_3(t)v_3(s)$\\ \rule{15mm}{0pt}$+\mathds{1}_{t>s}(t,s)v_1(t)v_2(s)
			$} 
			\\[6mm]
			\rule{7pt}{0pt}\parbox{50mm}{%
			$\mathds{1}_{t<s}(t,s)v_2(t)v_1(s)$\\ \rule{15mm}{0pt}$+\mathds{1}_{t>s}(t,s)v_3(t)v_3(s)
			$} &
			\parbox{49mm}{%
			$\mathds{1}_{t<s}(t,s)v_2(t)v_3(s)$\\ \rule{15mm}{0pt}$+\mathds{1}_{t>s}(t,s)v_3(t)v_2(s)
			$} 
		\end{pmatrix}
	\]
	The adjoint $T_{ij}^*$ is the integral operator with kernel 
	\[
		\parbox{\textwidth}{${\displaystyle\mkern50mu
		T_{ij}^*:\quad \mathds{1}_{t>s}(t,s)v_j(t)v_i(s)
		}$}
	\]
	and the assertion follows.
\end{proof}

\begin{corollary}\thlab{Q119}
	Let $H=\smmatrix{h_1}00{h_2}$ be a diagonal Hamiltonian, and let $S_{21}$ 
	be the integral operator in $L^2(a,b)$ with kernel 
	\[
		\parbox{\textwidth}{${\displaystyle\mkern50mu
		S_{21}:\quad \mathds{1}_{t<s}(t,s)\sqrt{h_2(t)}\sqrt{h_1(s)}
		}$}
	\]
	Then for every operator ideal $\OI$ we have 
	\[
		\OpB{H}\in\OI\ \Leftrightarrow\ S_{21}\in\OI
		.
	\]
\end{corollary}
\begin{proof}
	\thref{Q117} gives
	\begin{equation}\label{Q118}
		\OpC{H}=-
		\begin{pmatrix}
			0 & S_{21}^*
			\\
			S_{21} & 0
		\end{pmatrix}
		.
	\end{equation}
\end{proof}

\noindent
For a bounded function $\psi$, we denote by $M_\psi$ the multiplication operator with $\psi$ on $L^2(a,b)$:
\[
	(M_\psi f)(t)\DE\psi(t)f(t),\quad \|M_\psi\|=\|\psi\|_\infty
	.
\]

\begin{proof}[Proof of \thref{Q113}]
	It holds that $h_1=v_1^2+v_3^2$ and $h_2=v_2^2+v_3^2$, and hence 
	\[
		0\leq v_1\leq\sqrt{h_1},\ 0\leq v_2\leq\sqrt{h_2},\quad |v_3|\leq\min\{\sqrt{h_1},\sqrt{h_2}\}
		.
	\]
	Thus the functions (quotients are understood as $0$ if their denominator vanishes) 
	\[
		\frac{v_1}{\sqrt{h_1}},\ \frac{v_2}{\sqrt{h_2}},\ \frac{v_3}{\sqrt{h_1}},\ \frac{v_3}{\sqrt{h_2}},\ 
		\psi_1\DE\frac{v_3}{\sqrt{h_1}+v_1},\ \psi_2\DE\frac{v_3}{\sqrt{h_2}+v_2}
		,
	\]
	are all bounded. We have 
	\begin{align*}
		& T_{31}=M_{v_3/\sqrt{h_2}}S_{21}M_{v_1/\sqrt{h_1}},\quad
		T_{23}=M_{v_2/\sqrt{h_2}}S_{21}M_{v_3/\sqrt{h_1}}
		,
		\\
		& T_{21}=M_{v_2/\sqrt{h_2}}S_{21}M_{v_1/\sqrt{h_1}},\quad
		T_{33}=M_{v_3/\sqrt{h_2}}S_{21}M_{v_3/\sqrt{h_1}}
		,
		\\[1mm]
		& S_{21}=M_{\psi_2}T_{33}M_{\psi_1}+M_{\psi_2}T_{31}+T_{23}M_{\psi_1}+T_{21}
		,
	\end{align*}
	where the last relation holds since, by a short computation, 
	\[
		\sqrt{h_j}=\Big(\frac{v_3}{\sqrt{h_j}+v_j}\Big)\cdot\,v_3+v_j,\quad j=1,2
		.
	\]
	We see that 
	\[
		S_{21}\in\OI\quad\Leftrightarrow\quad T_{21},T_{23},T_{31},T_{33}\in\OI
		.
	\]
	From this it readily follows that $\OpB{\Diag H}\in\OI$ implies $\OpB{H}\in\OI$. 

	Assume that the weak Matsaev Theorem holds for $\OI$. If $\OpB{H}\in\OI$, then 
	\[
		\Re T_{31},\Re T_{23},\Re T_{21}+\Re T_{33}\in\OI
		.
	\]
	The operator $\Re T_{33}$ is one-dimensional, hence certainly belongs to $\OI$, and it follows that also $\Re T_{21}\in\OI$.
	We conclude that $T_{23}$, $T_{31}$, $T_{21}$, and $T_{33}$, all belong to $\OI$. 
	From this $S_{21}\in\OI$, and in turn $\OpB{\Diag H}\in\OI$. 
\end{proof}

\section{Discreteness of the spectrum}

\thref{Q102} is shown using a minor extension of \cite[Theorem~3.2]{aleksandrov.janson.peller.rochberg:2002},
namely \thref{Q120} below. 
In order to formulate this result, we need some notation. 
Let $-\infty\leq a<b\leq\infty$, let $\kappa,\varphi\DF(a,b)\to\bb C$ be measurable functions 
such that $\kappa\in L^2(a,b)$ and $\mathds{1}_{(a,c)}\varphi\in L^2(a,b)$ for every $c\in(a,b)$.
Then the function $t\mapsto\|\mathds{1}_{(t,b)}\kappa\|^2$ is a nonincreasing surjection of $[a,b]$ onto $[0,\|\kappa\|^2]$. 
Hence, we can choose an increasing sequence $c_0\DE a<c_1<c_2<\ldots<b$ such that
$\|\mathds{1}_{(c_n,b)}\kappa\|^2=2^{-n}\|\kappa\|^2$, $n\in\bb N$. Note that this requirement is equivalent to 
\begin{equation}\label{Q124}
	\|\mathds{1}_{(c_{n-1},c_n)}\kappa\|^2=\Big(\frac 12\Big)^n\|\kappa\|^2,\quad n\in\bb N
	.
\end{equation}
Having chosen $c_n$, we denote
\begin{equation}\label{Q121}
	J_n\DE(c_{n-1},c_n),\qquad\omega_n\DE \|\mathds{1}_{J_n}\kappa\|\cdot\|\mathds{1}_{J_n}\varphi\|
	,\quad n\in\bb N
	.
\end{equation}
Explicitly, by \eqref{Q124},
\[
	\omega_n=\|\kappa\|\cdot 2^{-\frac n2}\Big(\int_{c_{n-1}}^{c_n}|\varphi(s)|^2\DD ds\Big)^{\frac 12},\quad n\in\bb N
	.
\]

\begin{theorem}\thlab{Q120}
	Let $-\infty\leq a<b\leq\infty$, let $\kappa,\varphi\DF(a,b)\to\bb C$ be measurable functions 
	with $\kappa\in L^2(a,b)$ and $\mathds{1}_{(a,c)}\varphi\in L^2(a,b)$, $c\in(a,b)$, and consider the 
	integral operator $T$ on $L^2(a,b)$ with kernel \eqref{Q110}. 
	Then 
	\begin{center}
	\begin{tikzcd}[column sep=nixx]
		\parbox{21mm}{$T$ is compact \raisebox{-8pt}{\rule{0pt}{20pt}}} \arrow[Rightarrow]{rr}
		& \rule{0pt}{0pt} & \parbox{26mm}{ \raisebox{-8pt}{\rule{0pt}{20pt}} $\Re T$ is compact} \arrow[Rightarrow]{ld}\\
		& \parbox{24mm}{\rule{0pt}{12pt}$\lim_{n\to\infty}\omega_n=0$} \arrow[Rightarrow]{lu} &
	\end{tikzcd}
	\\[-3mm] \hfill
	\end{center}
	where $\omega_n$ are as in \eqref{Q121}.
	\popQED\qed
\end{theorem}

\noindent
The proof of \thref{Q120} is nearly verbatim the same as the argument given in \cite{aleksandrov.janson.peller.rochberg:2002}. 
We therefore skip details from the main text; the reader can find the fully 
elaborated argument in Appendix~A.

Rewriting the sequential condition occuring from \thref{Q120} to a continuous one as stated in \thref{Q102} is elementary; 
details are deferred to Appendix~B.

\begin{lemma}\thlab{Q122}
	Letting notation be as in \thref{Q120}, we have 
	\[
		\lim_{n\to\infty}\omega_n=0
		\ \Leftrightarrow\ 
		\lim_{t\nearrow b}\|\mathds{1}_{(a,t)}\varphi\|\|\mathds{1}_{(t,b)}\kappa\|=0
		.
	\vspace*{-5mm}
	\]
	\popQED\qed
\end{lemma}

\noindent
Now \thref{Q102} follows easily. 

\begin{proof}[Proof of \thref{Q102}]
	\thref{Q120} implies that the weak Matsaev Theorem holds in the operator ideal $\OI=c_0$ of all compact operators 
	(remember that we do not distinguish between a concrete operator ideal and its sequence space).
	Hence the Independence Theorem applies, and together with \thref{Q119} and \thref{Q120} applied to $\Diag H$ we find 
	\begin{equation}\label{Q127}
		\OpB{H}\in\SI_\infty\ \Leftrightarrow\ \lim_{n\to\infty}\omega_n=0
		,
	\end{equation}
	where $\omega_n$ is buildt with $\kappa(t)\DE\sqrt{h_1(t)}$, $\varphi(t)\DE\sqrt{h_2(t)}$. 
	By \thref{Q122}
	\[
		\lim_{n\to\infty}\omega_n=0\ \Leftrightarrow\ 
		\lim_{t\nearrow b}\Big(\int_t^b h_1(s)\DD ds\cdot\int_a^t h_2(s)\DD ds\Big)=0
		,
	\]
	and the proof of \thref{Q102} is complete.
\end{proof}

\begin{remark}\thlab{Q123}
	Using the connection between Krein strings and diagonal canonical systems elaborated in 
	\cite{kaltenbaeck.winkler.woracek:bimmel}, the present \thref{Q102} yields a new proof of 
	the classical criterion \cite[Theorema~4,5]{kac.krein:1958} for a string to have discrete spectrum.
\end{remark}

\section{Summability and limit superior conditions}

A symmetrically normed ideal $\OI$ is a  (two-sided) operator ideal which is endowed 
with a norm $\|.\|_{\OI}$, such that
\begin{Itemize}
\item $(\OI,\|.\|_{\OI})$ is complete, 
\item ${\displaystyle
	\|ATB\|_{\OI}\leq\|A\|\cdot\|T\|_{\OI}\cdot\|B\|,\quad T\in\OI,\ A,B\in\mc B(\mc H),\ \mc H\text{ Hilbert space}
	}$,
\item $\|T\|_{\OI}=\|T\|$ for $T$ with $\dim\ran T=1$. 
\end{Itemize}
Basic examples of symmetrically normed ideal are the Schatten-von~Neumann ideals $\SI_p$, $1\leq p\leq\infty$. 

Our standard reference about symmetrically normed ideals is \cite{gohberg.krein:1965}; 
another classical reference is \cite{schatten:1970}.

Recall that an operator $T$ is called a Volterra operator, if it is compact and $\sigma(T)=\{0\}$. 

\begin{definition}\thlab{Q150}
	Let $\OI$ be a symmetrically normed ideal which is properly contained in $\SI_\infty$. 
	We say that Matsaev's Theorem holds in $\OI$, if the following statement is true.
	\begin{Itemize}
	\item Let $\mc H$ be a Hilbert space, and let $T$ be a Volterra operator in $\mc H$. 
		Then $\Re T\in\OI$ implies $T\in\OI$. 
	\end{Itemize}
\end{definition}

\noindent
Notice that an integral operator whose kernel has the form \eqref{Q110} has no nonzero eigenvalues. 
As a consequence of this and \thref{Q120}, we obtain the following fact 
(which also justifies our terminology introduced in \thref{Q111}). 

\begin{corollary}\thlab{Q151}
	Let $\OI\subsetneq\SI_\infty$ be a symmetrically normed ideal. If Matsaev's Theorem holds in $\OI$, 
	then also the weak Matsaev Theorem holds in $\OI$. 
	\popQED\qed
\end{corollary}

\noindent
Consequently, for every proper symmetrically normed ideal in which Matsaev's Theorem holds, the Independence Theorem applies. 

The characterisations stated in \thref{Q104,Q106} will be shown using another AJPR-type theorem which is a variant 
of \cite[Theorem~3.3]{aleksandrov.janson.peller.rochberg:2002} 
(proof details of this result are given in Appendix~A).

We use the notation introduced in Section~3, in particular recall \eqref{Q124} and \eqref{Q121}. 
Moreover, recall that an operator ideal $\OI$ is called fully symmetric, if for each two nonincreasing sequences of 
nonnegative numbers $(\alpha_n)_{n=1}^\infty$ and $(\beta_n)_{n=1}^\infty$ it holds that 
\[
	\Big((\alpha_n)_{n=1}^\infty\in\OI\ \wedge\ \forall n\in\bb N\DP \sum_{k=1}^n\beta_k\leq\sum_{k=1}^n\alpha_k\Big)
	\quad\Longrightarrow\quad (\beta_n)_{n=1}^\infty\in\OI
\]
\begin{theorem}\thlab{Q152}
	Let $-\infty\leq a<b\leq\infty$, let $\kappa,\varphi\DF(a,b)\to\bb C$ be measurable functions 
	with $\kappa\in L^2(a,b)$ and $\mathds{1}_{(a,c)}\varphi\in L^2(a,b)$, $c\in(a,b)$, and consider the 
	integral operator $T$ on $L^2(a,b)$ with kernel \eqref{Q110}. 
	Moreover, let $\OI\subsetneq\SI_\infty$ be an operator ideal. 
	\begin{Itemize}
	\item If $\OI$ is fully symmetric, then $\Re T\in\OI$ implies $(\omega_n)_{n=1}^\infty\in\OI$.
	\item If $\OI$ is symmetrically normed and Matsaev's Theorem holds in $\OI$, 
		then $(\omega_n)_{n=1}^\infty\in\OI$ implies $T\in\OI$. 
	\end{Itemize}
	\popQED\qed
\end{theorem}

\noindent
To obtain the proof of \thref{Q104}, we use a particular class of symmetrically normed ideals. 

\begin{example}\thlab{Q154}
	Let $M:[0,\infty)\to[0,\infty)$ be a continuous increasing and convex function with $M(0)=0$, $\lim_{t\to\infty}M(t)=\infty$, 
	and $M(t)>0$, $t>0$. Assume that $\limsup_{t\searrow 0}\frac{M(2t)}{M(t)}<\infty$, and that $M$ is normalised by $M(1)=1$.
	The Orlicz space $\Orl{M}$ is the symmetrically normed ideal 
	\begin{align*}
		& \Orl{M}\DE\Big\{(\alpha_n)_{n=1}^\infty\in c_0\DS \sum_{n=1}^\infty M(|\alpha_n|)<\infty\Big\}
		,
		\\
		& \|(\alpha_n)_{n=1}^\infty\|_{\Orl{M}}\DE
		\inf\Big\{\beta>0\DS\sum_{n=1}^\infty M\Big(\frac{|\alpha_n|}\beta\Big)\leq 1\Big\}
		.
	\end{align*}
	Orlicz ideals are separable; in fact the unit vectors $e_j\DE(\delta_{nj})_{n=1}^\infty$, $j\in\bb N$, 
	form an unconditional basis in $\Orl{M}$. 
	For a systematic treatment of this type of sequence spaces we refer to \cite{maligranda:1989} 
	and \cite[Section~4.a]{lindenstrauss.tzafriri:1977}. 
\end{example}

\begin{remark}\thlab{Q156}
	Given a growth function $\ms g$ with order $\rho_{\ms g}>1$, set 
	\[
		M(t)\DE\frac 1{\ms g(\frac 1t)}
		.
	\]
	In general, $M$ will not be convex. However, based on \cite[Theorem~1.3.3]{bingham.goldie.teugels:1989}, 
	\cite[Proposition~1.22]{lelong.gruman:1986}, we always find an equivalent growth function $\ms g_1$, 
	i.e., one with $\lim_{t\to\infty}\frac{\ms g_1(t)}{\ms g(t)}=1$, such that the corresponding $M_1$ 
	satisfies all requirements made in \thref{Q154}. Then 
	\begin{equation}\label{Q158}
		\Orl{M_1}=
		\Big\{(\alpha_n)_{n=1}^\infty\in c_0\DS \sum_{n=1}^\infty \frac 1{\ms g\big(|\alpha_n|^{-1}\big)}<\infty\Big\}
		,
	\end{equation}
	and we may say that $\ms g$ induces an Orlicz ideal. 
\end{remark}

\noindent
Rewriting the sequential condition occuring from \thref{Q120} to a continuous one requires some technique about Orlicz spaces. 
Details are given in Appendix~B.

\begin{lemma}\thlab{Q155}
	Letting notation be as above, we have 
	\[
		\sum_{n=1}^\infty \frac 1{\ms g(\omega_n^{-1})}<\infty
		\ \Leftrightarrow\ 
		\int_a^b\Big[
		\ms g\Big(\big(\|\mathds{1}_{(a,t)}\varphi\|\|\mathds{1}_{(t,b)}\kappa\|\big)^{-1}\Big)\Big]^{-1}
		\mkern-5mu\cdot\mkern1mu
		\frac{|\kappa(t)|^2\DD dt}{\|\mathds{1}_{(t,b)}\kappa\|^2}<\infty
		.
	\]
	\popQED\qed
\end{lemma}

\begin{proof}[Proof of \thref{Q104}]
	The left and right sides of the equivalence asserted in \thref{Q104} do not change their truth value 
	when we pass from the given growth function $\ms g$ to an equivalent one. Hence, we may assume without loss 
	of generality that $\ms g$ gives rise to an Orlicz space $\Orl{M}$ as in \eqref{Q158}.

	Since $\rho_{\ms g}>1$, we can apply \cite[Footnote~12,p.139]{gohberg.krein:1967} and conclude that 
	Matsaev's Theorem holds in $\Orl{M}$. 
	By \thref{Q151} the weak Matsaev Theorem holds in $\Orl{M}$, and hence the Independence Theorem applies. 
	Clearly $\Orl{M}$ is fully symmetric, and \thref{Q113} combined with \thref{Q119} and \thref{Q152} yields 
	\begin{equation}\label{Q157}
		\OpB{H}\in\Orl{M}\ \Leftrightarrow\ \sum_{n=1}^\infty\frac 1{\ms g(\omega_n^{-1})}<\infty
		,
	\end{equation}
	where $\omega_n$ is buildt with $\kappa(t)\DE\sqrt{h_1(t)}$, $\varphi(t)\DE\sqrt{h_2(t)}$. 
	By \thref{Q155}
	\begin{multline*}
		\sum_{n=1}^\infty\frac 1{\ms g(\omega_n^{-1})}<\infty
		\ \Leftrightarrow\quad
		\\
		\int_a^b\bigg[
		\ms g\bigg(\Big(\int_t^b h_1(s)\DD ds\cdot\int_a^t h_2(s)\DD ds\Big)^{-\frac 12}\bigg)
		\bigg]^{-1}
		\mkern-5mu\cdot\mkern1mu
		\frac{h_1(t)\DD dt}{\int_t^b h_1(s)\DD ds}<\infty
		,
	\end{multline*}
	and the proof of \thref{Q104} is complete.
\end{proof}

\noindent
To obtain the proof of \thref{Q106}, we use another particular class of symmetrically normed ideals. 
In the following we denote for a sequence $(\alpha_n)_{n=1}^\infty\in c_0$,
by $(\alpha_n^*)_{n=1}^\infty$ its nonnegative nonincreasing rearrangement, i.e., 
the sequence made up of the elements $|\alpha_n|$ arranged nonincreasingly.

\begin{example}\thlab{Q153}
	Let $(\pi_n)_{n=1}^\infty$ be a nonincreasing positive sequence with $\pi_1=1$, $\lim_{n\to\infty}\pi_n=0$, 
	and $\sum_{n=0}^\infty\pi_n=\infty$. 
	The Lorentz space $\LI{\pi}$ is the symmetrically normed ideal 
	\begin{align*}
		& \LI{\pi}\DE\Big\{(\alpha_n)_{n=1}^\infty\in c_0\DS 
		\sup_{n\in\bb N}\Big(\sum_{k=1}^n\alpha_k^*\Big/\sum_{k=1}^n\pi_k\Big)<\infty\Big\}
		,
		\\
		& \|(\alpha_n)_{n=1}^\infty\|_{\LI{\pi}}\DE
		\sup_{n\in\bb N}\Big(\sum_{k=1}^n\alpha_k^*\Big/\sum_{k=1}^n\pi_k\Big)
		.
	\end{align*}
	Lorentz spaces may or may not be separable, and we denote by $\SLI{\pi}$ the separable part of $\LI{\pi}$. 
	I.e., $\SLI{\pi}$ is the closure in $\LI{\pi}$ of the linear subspace of all finite rank operators. 
	For this type of sequence spaces we refer to \cite[Theorem~III.14.1]{gohberg.krein:1965} or 
	\cite[Example~1.2.7]{lord.sukochev.zanin:2013}.
\end{example}

\noindent
This type of symmetrically normed ideals correspond to the consideration of limit superior conditions. 
Let $\ms g$ be a growth function with $\rho_{\ms g}>1$, and set $\pi_n\DE\frac 1{\ms g^{-1}(n)}$ where $\ms g^{-1}$ 
is the inverse function of $\ms g$. 
Then
\begin{align*}
	\LI{\pi}= &\, 
	\big\{(\alpha_n)_{n=1}^\infty\DS \limsup_{n\to\infty}\alpha_n^*\ms g^{-1}(n)<\infty\big\}
	,
	\\
	\SLI{\pi}= &\, \big\{(\alpha_n)_{n=1}^\infty\DS \lim_{n\to\infty}\alpha_n^*\ms g^{-1}(n)=0\big\}
	,
\end{align*}
cf.\ \cite[Theorem~III.14.2]{gohberg.krein:1965}.

\begin{proof}[Proof of \thref{Q106}]
	Since $\rho_{\ms g}>1$, we can apply \cite[Theorem~III.9.1]{gohberg.krein:1967} and conclude that Matsaev's Theorem holds 
	in $\LI{\pi}$ and in $\SLI{\pi}$. 
	By \thref{Q151} the weak Matsaev Theorem holds in $\Orl{M}$, and hence the Independence Theorem applies. 
	Moreover, $\LI{\pi}$ and $\SLI{\pi}$ are fully symmetric, cf.\ \cite{gohberg.krein:1965}.
	Now \thref{Q113} combined with \thref{Q119} and \thref{Q152} yields 
	\begin{equation}\label{Q159}
	\begin{aligned}
		& \OpB{H}\in\LI{\pi}\ \Leftrightarrow\ \limsup_{n\to\infty}\omega_n^*\ms g^{-1}(n)<\infty
		,
		\\
		& \OpB{H}\in\SLI{\pi}\ \Leftrightarrow\ \lim_{n\to\infty}\omega_n^*\ms g^{-1}(n)=0
		,
	\end{aligned}
	\end{equation}
	where $\omega_n$ is buildt with $\kappa(t)\DE\sqrt{h_1(t)}$, $\varphi(t)\DE\sqrt{h_2(t)}$. 

	Matching notation shows that these numbers $\omega_n$ are just the same as the numbers written in 
	\thref{Q106}. By a property of growth functions, it holds that 
	\begin{align*}
		& \limsup_{n\to\infty}\omega_n^*\ms g^{-1}(n)<\infty\ \Leftrightarrow\ 
		\limsup_{n\to\infty}\frac n{\ms g((\omega_n^*)^{-1})}<\infty
		,
		\\
		& \lim_{n\to\infty}\omega_n^*\ms g^{-1}(n)=0\ \Leftrightarrow\ 
		\lim_{n\to\infty}\frac n{\ms g((\omega_n^*)^{-1})}=0
		,
	\end{align*}
	and the proof is complete.
\end{proof}

\section{Bounded invertibility, normalisation, and examples}
\subsection{Bounded invertibility}

The present method also yields a condition for the model operator $\OpA{H}$ to be boundedly invertible. 

\begin{theorem}\thlab{Q172}
	Let $H=\smmatrix{h_1}{h_3}{h_3}{h_2}$ be a Hamiltonian on $[a,b)$ and assume that $\int_a^b h_1(s)\DD ds<\infty$. 
	Then $0\notin\sigma(\OpA{H})$ if and only if 
	\[
		\limsup_{t\nearrow b}\Big(\int_t^b h_1(s)\DD ds\cdot\int_a^t h_2(s)\DD ds\Big)<\infty
		.
	\]
\end{theorem}

\noindent
Remember here \thref{Q103}.

This result implies \cite[Theorem~1.5]{remling.scarbrough:1811.07067v1}.

Again we use the notation introduced in Section~3, in particular recall \eqref{Q124} and \eqref{Q121}. 
The proof of \thref{Q172} is based on the following AJPR-type theorem, which is a variant of 
\cite[Theorem~3.1]{aleksandrov.janson.peller.rochberg:2002} (proof details are given in Appendix~A). 

\begin{theorem}\thlab{Q170}
	Let $-\infty\leq a<b\leq\infty$, let $\kappa,\varphi\DF(a,b)\to\bb C$ be measurable functions 
	with $\kappa\in L^2(a,b)$ and $\mathds{1}_{(a,c)}\varphi\in L^2(a,b)$, $c\in(a,b)$, and consider the 
	integral operator $T$ on $L^2(a,b)$ with kernel \eqref{Q110}. 
	Then 
	\begin{center}
	\begin{tikzcd}[column sep=nixx]
		\parbox{21mm}{$T$ is bounded \raisebox{-8pt}{\rule{0pt}{20pt}}} \arrow[Rightarrow]{rr}
		& \rule{0pt}{0pt} & \parbox{26mm}{ \raisebox{-8pt}{\rule{0pt}{20pt}} $\Re T$ is bounded} \arrow[Rightarrow]{ld}\\
		& \parbox{24mm}{\rule{0pt}{12pt}$\sup_{n\in\bb N}\omega_n<\infty$} \arrow[Rightarrow]{lu} &
	\end{tikzcd}
	\\[-3mm] \hfill
	\end{center}
	where $\omega_n$ are as in \eqref{Q121}.
	\popQED\qed
\end{theorem}

\noindent
Rewriting the sequential condition occuring from \thref{Q170} to a continuous one is elementary; 
again details are deferred to Appendix~B.

\begin{lemma}\thlab{Q171}
	Letting notation be as above, we have 
	\[
		\sup_{n\to\infty}\omega_n<\infty
		\ \Leftrightarrow\ 
		\limsup_{t\nearrow b}\|\mathds{1}_{(a,t)}\varphi\|\|\mathds{1}_{(t,b)}\kappa\|<\infty
		.
	\vspace*{-5mm}
	\]
	\popQED\qed
\end{lemma}

\begin{proof}[Proof of \thref{Q172}]
	\thref{Q170} implies that the weak Matsaev Theorem holds in the operator ideal $\OI=\ell^\infty$ of all bounded operators.
	Hence the Independence Theorem applies, and together with \thref{Q119} and \thref{Q170} applied to $\Diag H$ we find 
	\begin{equation}\label{Q181}
		\OpB{H}\in\mc B(L^2(H))\ \Leftrightarrow\ \limsup_{n\to\infty}\omega_n<\infty
		,
	\end{equation}
	where $\omega_n$ is buildt with $\kappa(t)\DE\sqrt{h_1(t)}$, $\varphi(t)\DE\sqrt{h_2(t)}$. 
	By \thref{Q171}
	\[
		\sup_{n\to\infty}\omega_n<\infty\ \Leftrightarrow\ 
		\limsup_{t\nearrow b}\Big(\int_t^b h_1(s)\DD ds\cdot\int_a^t h_2(s)\DD ds\Big)<\infty
		,
	\]
	and the proof of \thref{Q172} is complete.
\end{proof}

\subsection{The normalisation $\bm{\int_a^b h_1(s)\DD ds<\infty}$}

In this section we provide the arguments announced in \thref{Q103}. 
Denote by $T_{\min}(H)$ and $T_{\max}(H)$ the minimal and maximal operators induced by the equation 
\eqref{Q101}, cf.\ \cite[Section~3]{hassi.snoo.winkler:2000}.
First observe that, in the cases of present interest, the space $L^2(H)$ always contains some constant.

\begin{lemma}\thlab{Q174}
	Assume that $0\notin\sigma_{\ess}(\OpA{H})$. 
	Then there exists $\phi\in\bb R$ such that 
	\begin{equation}\label{Q182}
		\binom{\cos\phi}{\sin\phi}\in L^2(H)
		.
	\end{equation}
\end{lemma}

\noindent
This follows from \cite[Theorem~3.8(b)]{remling:2018} (use $t=0$); for the convenience of the reader 
we recall the argument. 

\begin{proof}[Proof of \thref{Q174}]
	Since $0\notin\sigma_{\ess}(\OpA{H})$, $0$ is a point of regular type for $T_{\min}(H)$. Thus there 
	exists a selfadjoint extension $\tilde A$ of $T_{\min}(H)$ such that $0\in\sigma_{{\rm p}}(\tilde A)$
	(see, e.g., \cite[Propositions~3.3~and~3.5]{gorbachuk.gorbachuk:1997}), and it follows that 
	$\ker T_{\max}(H)\neq\{0\}$. This kernel, however, consists of all constant functions in $L^2(H)$. 
\end{proof}

\noindent
To achieve the normalisation $\int_a^b h_1(s)\DD ds<\infty$, equivalently, $\phi=0$ in \eqref{Q182}, 
one uses rotation isomorphisms. 

\begin{definition}\thlab{Q175}
	Let $\alpha\in\bb R$, and denote 
	\[
		N_\alpha\DE 
		\begin{pmatrix}
			\cos\alpha & \sin\alpha
			\\
			-\sin\alpha & \cos\alpha
		\end{pmatrix}
		.
	\]
	\begin{Enumerate}
	\item For a Hamiltonian $H$ defined on some interval $[a,b)$, we set 
		\[
			(\Rot{\alpha}{H})(t)\DE N_\alpha H(t)N_\alpha^{-1},\quad t\in[a,b)
			.
		\]
	\item For a $2$-vector valued function defined on some interval $[a,b)$, we set 
		\[
			(\omega_\alpha f)(t)\DE N_\alpha f(t),\quad t\in[a,b)
			.
		\]
	\end{Enumerate}
\end{definition}

\begin{remark}\thlab{Q176}
	The following facts hold (see, e.g., \cite[p.263]{kaltenbaeck.woracek:p5db}):
	\begin{Itemize}
	\item $\Rot{\alpha}{H}$ is a Hamiltonian,
	\item $\omega_\alpha$ induces an isometric isomorphism of $L^2(H)$ onto $L^2(\Rot{\alpha}{H})$,
	\item ${\displaystyle
			T_{\min}(\Rot{\alpha}{H})\circ\omega_\alpha=\omega_\alpha\circ T_{\min}(H)
		}$.
	\end{Itemize}
	Consequently, the Hamiltonians $H$ and $\Rot{\alpha}{H}$ will share all operator theoretic properties. 
\end{remark}

\noindent
In the context of diagonalisation, making a normalising rotation is inevitable. The reason being the following fact. 

\begin{lemma}\thlab{Q177}
	For a Hamiltonian $H$ there exists at most one angle $\alpha$ modulo $\frac\pi 2$ such that 
	$L^2(\Diag\Rot{\alpha}{H})$ contains a nonzero constant. 
\end{lemma}
\begin{proof}
	Let $\alpha\in[0,\pi)$, write 
	\[
		\tilde H=\begin{pmatrix} \tilde h_1 & \tilde h_3 \\ \tilde h_3 & \tilde h_2 \end{pmatrix}
		\DE\Rot{\alpha}H
		,
	\]
	and assume that $\phi\in\bb R$ with
	\[
		\xi_\phi\DE\binom{\cos\phi}{\sin\phi}\in L^2(\Diag\tilde H)
		.
	\]
	Since
	\[
		\xi_\phi^*\big(\Diag\tilde H\big)\xi_\phi=\tilde h_1\cos^2\phi+\tilde h_2\sin^2\phi
		,
	\]
	and since $H$ (and with it $\tilde H$ and $\Diag\tilde H$) is in limit point case, we conclude that 
	$\phi\in\{0,\frac\pi 2\}$. Thus, either $\binom 10$ or $\binom 01$ belongs to $L^2(\Diag\tilde H)$, 
	and hence to $L^2(\tilde H)$. From this we obtain that either $\xi_\alpha$ or $\xi_{\alpha+\frac\pi 2}$ belongs to $L^2(H)$. 
	There exists at most one angle $\phi$ modulo $\pi$ such that $\xi_\phi\in L^2(H)$, and therefore $\alpha$ 
	is uniquely determined modulo $\frac\pi2$.
\end{proof}

\subsection{Discussion of \thref{Q109,Q132}}

For Hamiltonians of a particularly simple form the conditions given in \thref{Q102,Q104,Q106} can be evaluated. 
We consider Hamiltonians which are defined on the interval $(0,1)$, 
where $h_1(t)=1$ a.e., and where $h_2(t)$ grows sufficiently regularly towards the singular endpoint $1$ in the 
following sense.

\begin{definition}\thlab{Q178}
	\hfill
	\begin{Enumerate}
	\item A function $\psi\DF [1,\infty)\to(0,\infty)$ is called regularly varying with index $\rho_\psi\in\bb R$, if 
		it is measurable and 
		\[
			\forall k>0\DP \lim_{x\to\infty}\frac{\psi(kx)}{\psi(x)}=k^{\rho_\psi}
			.
		\]
	\item We call a function $\varphi\DF [0,1)\to(0,\infty)$ regularly varying at $1$ with index $\rho\in\bb R$, if the function 
		\[
			\psi(x)\DE\varphi\Big(\frac{x-1}x\Big)\DF [1,\infty)\to(0,\infty)
		\]
		is regularly varying with index $\rho$.
	\end{Enumerate}
\end{definition}

\begin{lemma}\thlab{Q179}
	Let $\varphi\DF[0,1)\to(0,\infty)$ be continuous and regularly varying at $1$ with index $\rho>0$, 
	and set $\kappa(t)\DE 1$, $t\in[0,1)$. 
	Then the numbers $\omega_n$ constructed in \eqref{Q121} satisfy\footnote{%
		We write $\alpha_n\asymp\beta_n$, if there exist $c_1,c_2>0$ such that 
		$c_1\alpha_n\leq\beta_n\leq c_2\alpha_n$, $n\in\bb N$. 
		}
	\[
		\omega_n\asymp 2^{-n}\varphi\big(1-2^{-n}\big)
		.
	\]
\end{lemma}
\begin{proof}
	We have $\|\mathds{1}_{(t,b)}\kappa\|^2=b-t$, and hence the sequence $(c_n)_{n=0}^\infty$ is given as 
	\[
		c_n=1-2^{-n},\quad n=0,1,2,\ldots
		.
	\]
	Since $\varphi$ is continuous, we find $t_n\in J_n$ with 
	\[
		\omega_n=2^{-\frac n2}\Big(\int_{J_n}|\varphi(s)|^2\DD ds\Big)^{\frac 12}
		=2^{-\frac n2}(c_n-c_{n-1})^{\frac 12}\varphi(t_n)=2^{-n}\varphi(t_n)
		.
	\]
	Set 
	\[
		\psi(x)\DE\varphi\Big(\frac{x-1}x\Big),\quad x_n\DE\frac 1{1-c_n},\ y_n\DE\frac 1{1-t_n}
		.
	\]
	Since $\frac{x_{n-1}}{x_n}=\frac 12$, we have $y_n=k_nx_n$ with $k_n\in[\frac 12,1]$. 
	From the Uniform Convergence Theorem, see e.g.\ \cite[Theorem~1.5.2]{bingham.goldie.teugels:1989}, 
	we obtain that 
	\[
		\Big(\frac 14\Big)^\rho\leq\frac{\psi(y_n)}{\psi(x_n)}\leq\Big(\frac 54\Big)^\rho
		,\quad n\text{ sufficiently large}
		.
	\]
	Passing back to $\varphi,c_n,t_n$, this yields $\varphi(t_n)\asymp\varphi(c_n)$.
\end{proof}

\begin{proof}[Proof of \thref{Q109}]
	We apply \thref{Q179} with the function $\varphi(t)\DE\sqrt{h_2(t)}$. This is justified, since the corresponding function 
	$\psi$ is
	\[
		\psi(x)=x^{\frac\alpha 2}(1+\log x)^{-\frac{\alpha_1}2}(1+\log^+\log x)^{-\frac{\alpha_2}2}
		,
	\]
	and hence is regularly varying with index $\frac\alpha 2$. 
	Therefore the numbers $\omega_n$ which decide about the behaviour of the operator $S_{21}$ satisfy 
	\begin{align*}
		\omega_n\asymp &\, 2^{-n}\varphi(1-2^{-n})
		\\
		= &\, 2^{-n}\cdot 2^{n\frac\alpha 2}(1+\log 2^n)^{-\frac{\alpha_1}2}(1+\log^+\log 2^n)^{-\frac{\alpha_2}2}
		\\
		\asymp &\, 2^{n(\frac\alpha 2-1)}n^{-\frac{\alpha_1}2}(\log n)^{-\frac{\alpha_2}2}
		.
	\end{align*}
	Using this relation, the stated spectral properties of $H_{\alpha;\alpha_1,\alpha_2}$ follow immediately from the
	sequential characterisations given in \eqref{Q159}, \eqref{Q127}, \eqref{Q157}, and \eqref{Q181}.
	Let us go through the cases.
	\begin{Itemize}
	\item First of all the Krein-de~Branges formula implies
		\begin{equation}\label{Q180}
			\liminf_{n\to\infty}\frac n{|\lambda_n|}\geq\int_0^1\sqrt{h_2(s)}\DD ds>0
			.
		\end{equation}
	\item If $(\alpha>2)$ or $(\alpha=2,\alpha_1<0)$ or $(\alpha=2,\alpha_1=0,\alpha_2<0)$, then 
		$\lim_{n\to\infty}\omega_n=\infty$, and hence $0$ belongs to the essential spectrum. 
	\item If $(\alpha=2,\alpha_1=\alpha_2=0)$, then $\omega_n\asymp 1$, and hence the spectrum is not discrete, 
		but bounded invertibility takes place. 
	\item If $(\alpha<2)$ or $(\alpha=2,\alpha_1>0)$ or $(\alpha=2,\alpha_1=0,\alpha_2>0)$, then 
		$\lim_{n\to\infty}\omega_n=0$, and hence the spectrum is discrete. 
	\item If $(\alpha=2,\alpha_1>0)$, then the convergence exponent of $(\omega_n)_{n=1}^\infty$ equals 
		$\frac 2{\alpha_1}$, while in the case $(\alpha=2,\alpha_1=0,\alpha_2>0)$, 
		the convergence exponent of $(\omega_n)_{n=1}^\infty$ is infinite. 
		From this and \eqref{Q180} it follows that (for $\alpha=2$)
		\[
			\text{conv.exp.\ of }(|\lambda_n|)_{n=1}^\infty=
			\begin{cases}
				\infty &\hspace*{-3mm},\quad \alpha_1=0,\alpha_2>0,
				\\
				\frac 2{\alpha_1} &\hspace*{-3mm},\quad \alpha_1\in(0,2),\alpha_2\in\bb R,
				\\
				1 &\hspace*{-3mm},\quad \alpha_1\geq 2,\alpha_2\in\bb R.
			\end{cases}
		\]
	\item If $(\alpha=2,\alpha_1\in(0,2),\alpha_2\in\bb R)$ and $\ms g(r)\DE r^{\frac 2{\alpha_1}}(\log r)^\gamma$, then 
		\begin{align*}
			\ms g\big(\frac 1{\omega_n}\big)\asymp &\, \ms g\big(n^{\frac{\alpha_1}2}(\log n)^{\frac{\alpha_2}2}\big)
			\\
			= &\, \big[n^{\frac{\alpha_1}2}(\log n)^{\frac{\alpha_2}2}\big]^{\frac 2{\alpha_1}}
			\big[\log\big(n^{\frac{\alpha_1}2}(\log n)^{\frac{\alpha_2}2}\big)\big]^\gamma
			\\
			\asymp &\, n(\log n)^{\frac{\alpha_2}{\alpha_1}+\gamma}
			.
		\end{align*}
		This shows that for $\gamma=-\frac{\alpha_2}{\alpha_1}$ we have $n\cdot\ms g(\frac 1{\omega_n})^{-1}\asymp 1$. 
		Since the sequence $(\omega_n)_{n=1}^\infty$ is comparable to a monotone sequence, it follows that 
		\[
			0<\limsup_{n\to\infty}\frac n{\ms g((\omega_n^*)^{-1})}<\infty
			.
		\]
	\end{Itemize}
\end{proof}

\begin{proof}[Proof of \thref{Q132}]
	With a simple trick properties of $\mr H_{\alpha_1,\alpha_2}$ can be obtained from \thref{Q109}. 
	To explain this, we start in the reverse direction. Consider the Hamiltonian $H_{2;\alpha_1,\alpha_2}$, and set 
	\[
		\mr H_{2;\alpha_1,\alpha_2}(t)\DE
		\begin{pmatrix}
			1 & -m(t)
			\\
			-m(t) & m(t)^2
		\end{pmatrix}
		,\quad t\in[0,1)
		,
	\]
	where 
	\[
		m(t)\DE\int_0^t h_2(s)\DD ds
		.
	\]
	Moreover, let $q$ be the Weyl-coefficient of $H_{2;\alpha_1,\alpha_2}$ and $\mr q$ the one of 
	$\mr H_{2;\alpha_1,\alpha_2}$. Then, by \cite[Theorem~4.2]{kaltenbaeck.winkler.woracek:bimmel}, we have 
	\[
		q(z)=\frac 1z\mr q(z^2)
		.
	\]
	Thus the spectra of $\OpA{\mr H_{2;\alpha_1,\alpha_2}}$ and $\OpA{H_{2;\alpha_1,\alpha_2}}$ are together 
	discrete or not. If these spectra are discrete, then the convergence exponent of $\sigma(\OpA{H_{2;\alpha_1,\alpha_2}})$ 
	is twice the convergence exponent of $\sigma(\OpA{\mr H_{2;\alpha_1,\alpha_2}})$. This yields 
	\[
		\text{conv.exp.\ of }\sigma(\OpA{\mr H_{2;\alpha_1,\alpha_2}})=
		\begin{cases}
			\frac 1{\alpha_1} &\hspace*{-3mm},\quad \alpha_1\in(0,2)
			,
			\\
			\frac 12 &\hspace*{-3mm},\quad \alpha_1\geq 2
			.
		\end{cases}
	\]
	Integrating by parts gives 
	\[
		\lim_{t\nearrow 1}\frac{m(t)}{h_2(t)(1-t)}=1
		.
	\]
	The function $(h_2(t)(1-t))^2$ is again of the form \eqref{Q133} with $\alpha=2$, but with the parameters 
	$2\alpha_1$ and $2\alpha_2$ instead of $\alpha_1$ and $\alpha_2$. Thus the spectrum of $\OpA{\mr H_{\alpha_1,\alpha_2}}$ 
	has the same asymptotic behaviour as the spectrum of $\OpA{\mr H_{2;\frac{\alpha_1}2,\frac{\alpha_2}2}}$.
\end{proof}

%
%
\addtocontents{toc}{\contentsline{section}{{\bf Appendix}}{}}
\appendix
\makeatletter
\DeclareRobustCommand{\@seccntformat}[1]{%
  \def\temp@@a{#1}%
  \def\temp@@b{section}%
  \ifx\temp@@a\temp@@b
  Appendix\ \csname the#1\endcsname.\quad%
  \else
  \csname the#1\endcsname\quad%
  \fi
} 
\makeatother
\renewcommand{\thelemma}{\Alph{section}.\arabic{lemma}}
%
%

%
%
%
\section{Proof of AJPR-type theorems}

In this appendix we give a detailed proof of \thref{Q120,Q152,Q170}. 

Recall the relevant notation. 
Given a finite or infinite interval $(a,b)$, and measurable functions $\kappa,\varphi\DF(a,b)\to\bb C$
with $\kappa\in L^2(a,b)$ and $\mathds{1}_{(a,c)}\varphi\in L^2(a,b)$, $c\in(a,b)$, 
we consider the possibly unbounded integral operator $T$ on $L^2(a,b)$ acting as 
\[
	(Tf)(t)\DE \varphi(t)\int_t^bf(s)\qu{\kappa(s)}\DD ds,\quad t\in(a,b)
	,
\]
on its natural maximal domain.
Let $c_0\DE a<c_1<c_2<\ldots<b$ is a sequence with 
\begin{equation}\label{Q190}
	\|\mathds{1}_{(c_{n-1},c_n)}\kappa\|^2=\Big(\frac 12\Big)^n\|\kappa\|^2,\quad n\in\bb N
	,
\end{equation}
and we denote
\[
	J_n\DE(c_{n-1},c_n),\qquad\omega_n\DE \|\mathds{1}_{J_n}\kappa\|\cdot\|\mathds{1}_{J_n}\varphi\|
	,\ n\in\bb N
	.
\]
Moreover, let $\OI$ be an operator ideal (remember that we do not distinguish notationally between an 
operator ideal and its sequence space).
The task is to prove the implications in the triangle 
\begin{equation}\label{Q191}
\begin{tikzcd}[column sep=nixx]
	T\in\OI \arrow[Rightarrow]{rr}{\text{trivial}}
	& \rule{0pt}{0pt} & \Re T\in\OI \arrow[Rightarrow]{ld}\\
	& (\omega_n)_{n=1}^\infty\in\OI \arrow[Rightarrow]{lu} &
\end{tikzcd}
\end{equation}
in the following situations.
\begin{Enumerate}
\item For \thref{Q170}: $\OI=\ell^\infty$.
\item For \thref{Q120}: $\OI=c_0$.
\item For \thref{Q152}: $\OI\subsetneq\SI_\infty$, 
	assuming that $\OI$ is fully symmetric for the downwards implication, 
	and assuming that $\OI$ is symmetrically normed and Matsaev's Theorem holds in $\OI$ for the upwards implication. 
\end{Enumerate}

\begin{proof}[Proof of ``\,$\Re T\in\OI\Rightarrow(\omega_n)_{n=1}^\infty\in\OI$'']
\hfill\\
	Let $\OI$ be $\ell^\infty$, $c_0$, or a fully symmetric operator ideal, and assume that $\Re T\in\OI$. 
	Moreover, denote by $P_n$ the orthogonal projection $P_nf\DE\mathds{1}_{J_n}f$ of $L^2(a,b)$ onto its subspace $L^2(J_n)$. 

	Since $\OI$ is of one of the stated forms, we obtain that $\sum_{n=1}^\infty P_n(\Re T)P_{n+1}\in\OI$. 
	For $\OI=\ell^\infty$ or $\OI=c_0$, this is obvious. For $\OI$ being fully symmetric, it is a consequence 
	of \cite[Theorem~II.5.1]{gohberg.krein:1965}\/\footnote{%
		This is actually a variant of \cite[Theorem~II.5.1]{gohberg.krein:1965} which is easy to obtain in the present 
		situation since all spaces $L^2(J_n)$ have the same Hilbert space dimension. Choose unitary 
		operators $U_n\DF L^2(J_n)\to L^2(J_{n+1})$, let $S\DF L^2(a,b)\to L^2(a,b)$ be the block shift 
		\[
			S\DE
			\left(
			\begin{array}{ccccc}
				0 & & & & \\[3pt]
				U_1 & 0 & & & \\[3pt]
				& U_2 & 0 & & \\[1pt]
				& & \ddots & \ddots &
			\end{array}
			\right)
			\DF
			\begin{pmatrix} L^2(J_1)\\ \oplus\\ L^2(J_2)\\ \oplus\\ \vdots\end{pmatrix}
			\to
			\begin{pmatrix} L^2(J_1)\\ \oplus\\ L^2(J_2)\\ \oplus\\ \vdots\end{pmatrix}
			,
		\]
		and apply \cite[Theorem~II.5.1]{gohberg.krein:1965} to the operator $(\Re T)S$.
	}.

	Clearly, $P_nTP_{n+1}=(\textvisiblespace\,,\mathds{1}_{J_{n+1}}\kappa)\mathds{1}_{J_n}\varphi$. 
	The adjoint of $T$ is the integral operator with kernel 
	\[
		\parbox{\textwidth}{${\displaystyle\mkern50mu
		T^*:\quad \mathds{1}_{s<t}(t,s)\kappa(t)\qu{\varphi(s)}
		}$}
	\]
	and hence $P_nT^*P_{n+1}=0$. Together, 
	\[
		\sum_{n=1}^\infty P_n(\Re T)P_{n+1}=\frac 12\sum_{n=1}^\infty 2^{-1/2}\omega_n\cdot
		\Big(\textvisiblespace\,,\frac{\mathds{1}_{J_{n+1}}\kappa}{\|\mathds{1}_{J_{n+1}}\kappa\|}\Big)
		\frac{\mathds{1}_{J_n}\varphi}{\|\mathds{1}_{J_n}\varphi\|}
		,
	\]
	and hence $a_n\big(\sum_{n=1}^\infty P_n(\Re T)P_{n+1}\big)=2^{-\frac 12}\omega_n^*$. 
	We see that $(\omega_n)_{n=1}^\infty\in\OI$. 
\end{proof}

\noindent
The upwards implication in the triangle \eqref{Q191} is a bit more involved.

\begin{proof}[Proof of ``\,$(\omega_n)_{n=1}^\infty\in\OI\Rightarrow T\in\OI$'']
\hfill\\
	Let $\OI$ be $\ell^\infty$, $c_0$, or a symmetrically normed ideal in which Matsaev's Theorem holds, and 
	assume that $(\omega_n)_{n=1}^\infty\in\OI$. 
	Note that in every case $(\omega_n)_{n=1}^\infty$ is bounded. 
\begin{Steps}
\item 
	The crucial point is to handle the diagonal cell sum $\sum_{n=1}^\infty P_nTP_n$.
	Our aim is to show that this series converges to an operator in $\OI$. 

	The summand $P_nTP_n$ is the integral operator in $L^2(a,b)$ with kernel 
	\[
		\parbox{\textwidth}{${\displaystyle\mkern50mu
		P_nTP_n:\quad \mathds{1}_{t<s}(t,s)\mathds{1}_{J_n}(t)\mathds{1}_{J_n}(s)\varphi(t)\qu{\kappa(s)}
		}$}
	\]
	Since $\mathds{1}_{J_n}\kappa,\mathds{1}_{J_n}\varphi\in L^2(a,b)$, it is compact and 
	\begin{align*}
		\|P_nTP_n\|\leq &\, \Big(
		\int_a^b\int_a^b|\mathds{1}_{t<s}(t,s)\mathds{1}_{J_n}(t)\mathds{1}_{J_n}(s)\varphi(t)\qu{\kappa(s)}|^2\DD ds\DD dt
		\Big)^{\frac 12}
		\\
		\leq &\, \|\mathds{1}_{J_n}\kappa\|\|\mathds{1}_{J_n}\varphi\|=\omega_n
		.
	\end{align*}
	The sequence $(\omega_n)_{n=1}^\infty$ is bounded, and hence the series $\sum_{n=1}^\infty P_nTP_n$ converges strongly, 
	and its sum is a bounded operator with 
	\[
		\Big\|\sum_{n=1}^\infty P_nTP_n\Big\|\leq\|(\omega_n)_{n=1}^\infty\|_\infty
		.
	\]
	This settles the case that $\OI=\ell^\infty$. 
	If $\lim_{n\to\infty}\omega_n=0$, the series converges w.r.t.\ the operator norm and hence its sum is a compact operator. 
	This settles the case that $\OI=c_0$. 

	Consider the remaining case. Then, in particular, $\lim_{n\to\infty}\omega_n=0$. 
	Let $Q_0$ be the compact operator given by the Schmidt-series 
	\[
		Q_0\DE\sum_{n=1}^\infty\omega_n\cdot
		\Big(\textvisiblespace\,,\frac{\mathds{1}_{J_n}\kappa}{\|\mathds{1}_{J_n}\kappa\|}\Big)
		\frac{\mathds{1}_{J_n}\varphi}{\|\mathds{1}_{J_n}\varphi\|}
		.
	\]
	Then $a_n(Q_0)=\omega_n^*$, and hence $Q_0\in\OI$. 
	Since Matsaev's Theorem holds in $\OI$, the triangular truncation transformator $\mc C$, cf.\ \cite{gohberg.krein:1967}, 
	is defined on all of $\OI$ and maps $\OI$ boundedly into itself. Thus $\mc CQ_0\in\OI$, and 
	\[
		\mc CQ_0=\sum_{n=1}^\infty\mc C\big((\textvisiblespace\,,\mathds{1}_{J_n}\kappa)\mathds{1}_{J_n}\varphi\big)
		.
	\]
	However, $\mc C\big((\textvisiblespace\,,\mathds{1}_{J_n}\kappa)\mathds{1}_{J_n}\varphi\big)=P_nTP_n$. 
	Thus we have $\sum_{n=1}^\infty P_nTP_n\in\OI$. 
\item 
	The rest of the proof merely uses completeness.
	For $l\in\bb N$ let $Q_l$ be the compact operator given by the Schmidt-series 
	\[
		Q_l\DE\sum_{n=1}^\infty 2^{-\frac l2}\omega_n\cdot
		\Big(\textvisiblespace\,,\frac{\mathds{1}_{J_{n+l}}\kappa}{\|\mathds{1}_{J_{n+l}}\kappa\|}\Big)
		\frac{\mathds{1}_{J_n}\varphi}{\|\mathds{1}_{J_n}\varphi\|}
		.
	\]
	Then $Q_l\in\OI$ and $\|Q_l\|_{\OI}=2^{-l/2}\|(\omega_n)_{n=1}^\infty\|_{\OI}$. Hence, the 
	series $\sum_{l=1}^\infty Q_l$ converges w.r.t.\ $\|\textvisiblespace\|_{\OI}$ and its sum belongs to $\OI$. 

	A short computation using \eqref{Q190} shows that 
	\[
		Tf=\Big(\sum_{n=1}^\infty P_nTP_n\Big)f+\Big(\sum_{l=1}^\infty Q_l\Big)f,\quad 
		f\in L^2(a,b),\sup\supp f<b
		.
	\]
	Since $T$ is closed, it follows that $T=\sum_{n=1}^\infty P_nTP_n+\sum_{l=1}^\infty Q_l$, and we 
	conclude that $T\in\OI$. 
\end{Steps}
\end{proof}

\section{Sequential vs.\ continuous conditions}

In this section we give detailed proofs of \thref{Q122,Q155,Q171}. 
Recall the relevant notation:
We are given a finite or infinite interval $(a,b)$, and measurable functions $\kappa,\varphi\DF(a,b)\to\bb C$
with $\kappa\in L^2(a,b)$ and $\mathds{1}_{(a,c)}\varphi\in L^2(a,b)$, $c\in(a,b)$. 
Further, $c_0\DE a<c_1<c_2<\ldots<b$ is a sequence with 
\[
	\|\mathds{1}_{(c_n,b)}\kappa\|^2=\Big(\frac 12\Big)^n\|\kappa\|^2\text{, equivalently, }
	\|\mathds{1}_{(c_{n-1},c_n)}\kappa\|^2=\Big(\frac 12\Big)^n\|\kappa\|^2
	,
\]
and $J_n\DE(c_{n-1},c_n)$ and $\omega_n\DE \|\mathds{1}_{J_n}\kappa\|\cdot\|\mathds{1}_{J_n}\varphi\|$. 
Moreover, denote 
\[
	\Omega(t)\DE\|\mathds{1}_{(t,b)}\kappa\|\cdot\|\mathds{1}_{(a,t)}\varphi\|,\quad t\in(a,b)
	.
\]
The proof of \thref{Q122} and \thref{Q171} is simple. 

\begin{proof}[Proof of \thref{Q122,Q171}]
	A sequence $(\alpha_n)_{n=1}^\infty$ of nonnegative numbers is bounded (tends to $0$), if and only if the 
	sequence $\big(2^{-n}\sum_{k=1}^n2^k\alpha_k\big)_{n=1}^\infty$ is bounded (tends to $0$, respectively). 
	Applying this with 
	\[
		\alpha_n\DE 2^{-n}\int_{c_{n-1}}^{c_n}|\varphi(s)|^2\DD ds,\quad n\in\bb N
		,
	\]
	yields the assertion.
\end{proof}

\noindent
The proof of \thref{Q155} is based on dualising and requires some technique for Orlicz ideals. 
We start with one preparatory lemma.

Let a growth function $\ms g$ with $\rho_{\ms g}>1$ be given. Passing to an equivalent growth function and modifying 
$\ms g$ on a finite interval does not change the truth value of the left side of the asserted equivalence. 
Hence, we may assume w.l.o.g.\ that the function $M(t)\DE\ms g(t^{-1})^{-1}$ has all properties required in 
\thref{Q154}, cf.\ \thref{Q156}, and additionally that $M(t)\asymp t^{\rho_\ms g}$, $t\geq 1$. Note that, since $\rho_\ms g>1$, 
\[
	\lim_{t\searrow 0}\frac{M(t)}t=\lim_{r\to\infty}\frac r{\ms g(r)}=0,\quad 
	\lim_{t\to\infty}\frac{M(t)}t=\infty
	.
\]
Using the language of \cite[Chapter~8,p.47]{maligranda:1989} this means that $M$ belongs to the class $N$. 

\begin{lemma}\thlab{Q192}
	Set $I\DE\bb N$ and let $q\in(0,1)$. 
	\begin{Enumerate}
	\item Set $J\DE\{(n,k)\in I\times I\DS k\leq n\}$ and let $\Orl{M}$ denote the Orlicz space of sequences 
		indexed by $I$ or by $J$ depending on the context. 
		For a sequence $(\alpha_n)_{n\in I}$ define sequences $(\beta_{(n,k)})_{(n,k)\in J}$ and 
		$(\beta_{(n,k)}')_{(n,k)\in J}$ as 
		\[
			\beta_{n,k}\DE\alpha_nq^{n-k},\ \beta_{(n,k)}'\DE\alpha_kq^{n-k},\quad (n,k)\in J
			.
		\]
		If $(\alpha_n)_{n\in I}\in\Orl{M}$, then $(\beta_{(n,k)})_{(n,k)\in J},(\beta_{(n,k)}')_{(n,k)\in J}\in\Orl{M}$. 
		There exists a constant $C_1>0$ such that 
		\begin{multline*}
			\max\big\{\|(\beta_{(n,k)})_{(n,k)\in J}\|_{\Orl{M}},\|(\beta_{(n,k)}')_{(n,k)\in J}\|_{\Orl{M}}\big\}
			\\
			\leq C_1\|(\alpha_n)_{n\in I}\|_{\Orl{M}},\quad (\alpha_n)_{n\in I}\in\Orl{M}
			.
		\end{multline*}
	\item Consider a sequence $(\alpha_n)_{n\in I}$ with $(q^n\alpha_n)_{n\in I}\in\Orl{M}$, and 
		define a sequence $(\beta_n)_{n\in I}$ as
		\[
			\beta_n\DE\sum_{\substack{k\in I\\ k\leq n}}\alpha_k,\quad n\in I
			.
		\]
		Then $(q^n\beta_n)_{n\in I}\in\Orl{M}$. There exists a constant $C_2>0$ such that 
		\[
			\|(q^n\beta_n)_{n\in I}\|_{\Orl{M}}\leq C_2\|(q^n\alpha_n)_{n\in I}\|_{\Orl{M}}
			,\quad (q^n\alpha_n)_{n\in I}\in\Orl{M}
			.
		\]
	\end{Enumerate}
\end{lemma}
\begin{proof}
	For the proof of item \Enumref{1} let $(\alpha_n)_{n\in I}\in\Orl{M}$ with 
	$\|(\alpha_n)_{n\in I}\|_{\Orl{M}}\leq 1$ be given. 
	Then, in particular, $|\alpha_n|\leq 1$, $n\in I$. 
	Since $\rho_\ms g>1$, we have
	\[
		C\DE\sup_{0<\gamma,t\leq 1}\frac{M(\gamma t)}{\gamma M(t)}<\infty
		,
	\]
	cf.\ \cite[Proposition~7.4.1,Theorem~1.5.6]{bingham.goldie.teugels:1989}. Note that $C\geq 1$. Thus we can estimate 
	\begin{align*}
		\sum_{(n,k)\in J}M(|\beta_{n,k}|)= &\, \sum_{(n,k)\in J}M(|\alpha_n|q^{n-k})
		\leq \sum_{(n,k)\in J} Cq^{n-k}M(|\alpha_n|)
		\\
		= &\, C\sum_{n\in I}\Big(\underbrace{\sum_{\substack{k\in I\\ k\leq n}}q^{n-k}}_{\leq\frac 1{1-q}}\Big)M(|\alpha_n|)
		\leq \frac C{1-q}\sum_{n\in I}M(|\alpha_n|)\leq\frac C{1-q}
		.
	\end{align*}
	Since $|\beta_{n,k}|\leq|\alpha_n|\leq 1$ and $\frac{C^2}{1-q}\geq 1$, it follows that 
	\[
		\sum_{(n,k)\in J}M\Big(\frac{|\beta_{n,k}|}{C^2/(1-q)}\Big)\leq\sum_{(n,k)\in J} C\frac{1-q}{C^2}M(|\beta_{n,k}|)
		\leq 1
		.
	\]
	This shows that $\|(\beta_{(n,k)})_{(n,k)\in J}\|_{\Orl{M}}\leq\frac{C^2}{1-q}$. 

	The sequence $(\beta_{(n,k)}')_{(n,k)\in J}$ is handled in the same way. Namely 
	\begin{align*}
		\sum_{(n,k)\in J}M(\beta_{n,k}')= &\, \sum_{(n,k)\in J}M(|\alpha_k|q^{n-k})
		\leq \sum_{(n,k)\in J} Cq^{n-k}M(|\alpha_k|)
		\\
		= &\, C\sum_{k\in I}\Big(\underbrace{\sum_{\substack{n\in I\\ n\geq k}}q^{n-k}}_{=\frac 1{1-q}}\Big)M(|\alpha_k|)
		=\frac C{1-q}\sum_{k\in I}M(|\alpha_k|)\leq\frac C{1-q}
		,
	\end{align*}
	from which we again obtain that $\|(\beta_{(n,k)}')_{(n,k)\in J}\|_{\Orl{M}}\leq\frac{C^2}{1-q}$. 

	The proof of \Enumref{2} is based on dualising.
	For each $\gamma>0$ we have 
	\[
		\lim_{t\searrow 0}\frac{M(\gamma t)}{M(t)}=\lim_{r\to\infty}\frac{\ms g(\gamma r)}{\ms g(r)}=\gamma^{\rho_\ms g}
		,
	\]
	cf.\ \cite[Proposition~7.4.1]{bingham.goldie.teugels:1989}.
	Hence the Orlicz indices of $M$ at $0$, cf.\ \cite[Chapter~11,p.84]{maligranda:1989}, are both equal to $\rho_\ms g$. 
	Let $M^*$ be the Orlicz function complementary to $M$, cf.\ \cite[Chapter~8,p.48]{maligranda:1989}. 
	Then both Orlicz indices of $M^*$ are equal to $\frac{\rho_\ms g}{\rho_\ms g-1}$, 
	cf.\ \cite[Corollary~11.6]{maligranda:1989}. From \cite[Theorem~11.13]{maligranda:1989} we obtain 
	\[
		C^*\DE\sup_{0<\gamma,t\leq 1}\frac{M^*(\gamma t)}{\gamma M^*(t)}<\infty
		.
	\]
	Now let $(\sigma_n)_{n\in I}\in\Orl{M^*}$. Then we can use \Enumref{1}, the H\"older inequality 
	\cite[Chapter~8,Corollary~3]{maligranda:1989}, and the relation \cite[Theorem~1.1]{maligranda:1989} between Amemiya- and 
	Luxemburg norms, to estimate
	\begin{align*}
		\Big|\sum_{n\in I}\sigma_n\cdot q^n\beta_n\Big|\leq &\, 
		\sum_{n\in I} |\sigma_n|q^n\sum_{\substack{k\in I\\ k\leq n}}|\alpha_k|=
		\sum_{(n,k)\in J} |\sigma_n|(\sqrt q)^{n-k}\cdot\big(q^k|\alpha_k|\big)(\sqrt q)^{n-k}
		\\
		\leq &\, 2\big\|\big(\sigma_n\cdot(\sqrt q)^{n-k}\big)_{(n,k)\in J}\big\|_{\Orl{M^*}}
		\big\|\big(q^k\alpha_k\cdot(\sqrt q)^{n-k}\big)_{(n,k)\in J}\big\|_{\Orl{M}}
		\\
		\leq &\, \frac{2C^2(C^*)^2}{(1-\sqrt q)^2}\|(\sigma_n)_{n\in I}\|_{\Orl{M^*}}\|(q^n\alpha_n)_{n\in I}\|_{\Orl{M}}
		.
	\end{align*}
	By \cite[Theorem~8.6]{maligranda:1989} it follows that $(q^n\beta_n)_{n\in I}\in\Orl{M}$ and 
	\[
		\|(q^n\beta_n)_{n\in I}\|_{\Orl{M}}\leq\frac{2C^2(C^*)^2}{(1-\sqrt q)^2}\|(q^n\alpha_n)_{n\in I}\|_{\Orl{M}}
		.
	\]
\end{proof}

\begin{proof}[Proof of \thref{Q155}]
	For $t\in J_n$ it holds that
	\[
		\Omega(t)\geq\|\mathds{1}_{(a,c_{n-1})}\varphi\|\|\mathds{1}_{(c_n,b)}\kappa\|
		\geq \|\mathds{1}_{J_{n-1}}\varphi\|\|\mathds{1}_{J_{n+1}}\kappa\|
		=\frac{\omega_{n-1}}2
		,
	\]
	and we can estimate 
	\begin{align*}
		\sum_{n=1}^\infty M\Big(\frac{\omega_n}2\Big)= &\, 
		\sum_{n=1}^\infty \bigg[M\Big(\frac{\omega_n}2\Big)\cdot
		\frac 1{\log 2}\int\limits_{J_{n+1}}\frac{|\kappa(t)|^2\DD dt}{\|\mathds{1}_{(t,b)}\kappa\|^2}
		\bigg]
		\\
		\leq &\, \sum_{n=1}^\infty \bigg[\frac 1{\log 2}\int\limits_{J_{n+1}}M(\Omega(t))\cdot 
		\frac{|\kappa(t)|^2\DD dt}{\|\mathds{1}_{(t,b)}\kappa\|^2}\bigg]
		\\
		\leq &\, \frac 1{\log 2}\int_a^b\frac 1{\ms g(\Omega(t)^{-1})}\cdot
		\frac{|\kappa(t)|^2\DD dt}{\|\mathds{1}_{(t,b)}\kappa\|^2}
		.
	\end{align*}
	This shows that the implication ``$\Leftarrow$'' holds. 

	Conversely, we have for $t\in J_n$ 
	\[
		\Omega(t)\leq \|\mathds{1}_{(a,c_n)}\varphi\|\|\mathds{1}_{(c_{n-1},b)}\kappa\|
		\leq\Big(\sum_{k=1}^n\|\mathds{1}_{J_k}\varphi\|\Big)\cdot\Big(\frac 1{\sqrt 2}\Big)^{n-1}\|\kappa\|
		.
	\]
	Assume that $(\omega_n)_{n=1}^\infty\in\Orl{M}$. Since 
	$\omega_n=\|\kappa\|\big(\frac 1{\sqrt 2}\big)^n\cdot\|\mathds{1}_{J_n}\varphi\|$, we can apply 
	\thref{Q192}\Enumref{2} with the sequence $\alpha_n\DE\|\mathds{1}_{J_n}\varphi\|$, $n\in\bb N$. This shows that 
	\[
		\Big(\Big(\frac 1{\sqrt 2}\Big)^n\sum_{k=1}^n\|\mathds{1}_{J_k}\varphi\|\Big)_{n=1}^\infty\in\Orl{M}
		,
	\]
	and we obtain 
	\begin{align*}
		\int_a^b M(\Omega(t))\cdot \frac{|\kappa(t)|^2\DD dt}{\|\mathds{1}_{(t,b)}\kappa\|^2}
		= &\, 
		\sum_{n=1}^\infty\int\limits_{J_{n+1}}M(\Omega(t))\cdot 
		\frac{|\kappa(t)|^2\DD d\mu(t)}{\|\mathds{1}_{(t,b)}\kappa\|_{L^2(\mu)}^2}
		\\
		\leq &\, \log 2\sum_{n=1}^\infty 
		M\Big(\|\kappa\|\Big(\frac 1{\sqrt 2}\Big)^{n-1}\sum_{k=1}^n\|\mathds{1}_{J_k}\varphi\|\Big)
		<\infty
		.
	\end{align*}
\end{proof}

\section{I.S.Kac's compactness theorem}

Let us recall \cite[Theorem~1]{kac:1995}, which is stated as the main result in this paper.
Unfortunately, Kac's original proofs are not available, and we do not know any source where proofs are given. 

\begin{theorem}[\cite{kac:1995}]\thlab{Q141}
	Let $H=\smmatrix{h_1}{h_3}{h_3}{h_2}$ be a Hamiltonian on $[0,\infty)$ which is normalised such that 
	$\tr H(t)=1$ a.e., and denote $m_j(t)\DE\int_0^t h_j(t)\DD dt$. 
	For $K>0$ set 
	\begin{align*}
		A_K^+\! \DE &\,
		\bigg\{\lambda\!\in\!\bb R\!\setminus\!\{0\}\DS \!\limsup_{t\to\infty}\bigg(
		\int\limits_t^\infty h_1(s)e^{2\lambda m_3(s)} ds\cdot\!\!\int\limits_0^t h_2(s)e^{-2\lambda m_3(s)} ds
		\!\bigg)\leq\frac K{\lambda^2}\bigg\}
		,
		\\
		B_K^+ \DE &\, \bigg\{\lambda\!\in\!\bb R\!\setminus\!\{0\}\DS \!\limsup_{t\to\infty}\bigg(
		\int\limits_t^\infty h_2(s)e^{-2\lambda m_3(s)} ds\cdot\!\!\int\limits_0^t h_1(s)e^{2\lambda m_3(s)} ds
		\!\bigg)\leq\frac K{\lambda^2}\bigg\}
		.
	\end{align*}
	Then \Enumref{1}$\Rightarrow$\Enumref{2}$\Rightarrow$\Enumref{3}, where
	\begin{Enumerate}
	\item ${\displaystyle
		\sup\Big(\bigcup_{K<\frac 14}A_K^+\cup\bigcup_{K<\frac 14}B_K^+\Big)=+\infty,\ 
		\inf\Big(\bigcup_{K<\frac 14}A_K^+\cup\bigcup_{K<\frac 14}B_K^+\Big)=-\infty
		}$.
	\item $\sigma(\OpA{H})$ is discrete with\footnote{%
			This normalisation does not appear in \cite{kac:1995}. However, without it the statement is false.
		}
		$\int_0^\infty h_1(s)\DD ds<\infty$ or $\int_0^\infty h_2(s)\DD ds<\infty$.
	\item $A_1^+\cup B_1^+=\bb R\setminus\{0\}$.
	\end{Enumerate}
\end{theorem}

\noindent
In the following theorem we make the connection to our present work.

\begin{theorem}\thlab{Q140}
	Let notation be as in \thref{Q141}. Then the following items are equivalent.
	\begin{Enumerate}
	\item There exists $K>0$ such that $\inf A_K^+=-\infty$ and $\sup A_K^+=+\infty$. 
	\item $\sigma(\OpA{H})$ is discrete with $\int_a^b h_1(s)\DD ds<\infty$. 
	\item ${\displaystyle \bigcap_{K>0}A_K^+=\bb R\setminus\{0\}}$.
	\end{Enumerate}
	The analogous statement holds when $A_K^+$ is replaced by $B_K^+$ and $h_1$ by $h_2$. 
\end{theorem}

\noindent
Concerning the normalisation in item \Enumref{2} remember again \thref{Q105}.
Moreover, note that the implication ``\Enumref{3}$\Rightarrow$\Enumref{1}'' is trivial. 

\begin{proof}[Proof of ``\Enumref{1}$\Rightarrow$\Enumref{2}'']
	Choose $c\in(0,\infty)$ such that $h_2$ does not vanish almost everywhere on $[0,c]$, and consider the 
	function $F:[c,\infty)\times\bb R\to[0,\infty]$ defined as 
	\[
		F(t,\lambda)\DE \int_t^\infty h_1(s)e^{2\lambda m_3(s)}\DD ds\cdot\int_0^t h_2(s)e^{-2\lambda m_3(s)}\DD ds
		.
	\]
\begin{Steps}
\item 
	We show that for every finite interval $[\mu_1,\mu_2]\subseteq\bb R$ it holds that 
	\begin{equation}\label{Q148}
		F(t,\lambda)\leq\max\big\{
		F(t,\mu_1),F(t,\mu_2),F(t,\mu_1)^{\frac 12},F(t,\mu_2)^{\frac 12},F(t,\mu_1)F(t,\mu_2)
		\big\}
		.
	\end{equation}
	Let $\lambda\in(\mu_1,\mu_2)$, and write $\lambda=v\mu_1+(1-v)\mu_2$ with $v\in(0,1)$. 
	The H\"older inequality gives 
	\begin{align*}
		\int\limits_t^\infty h_1(s)e^{2\lambda m_3(s)}\DD ds= &\,
		\int\limits_t^\infty \Big(h_1(s)e^{2\mu_1 m_3(s)}\Big)^v\cdot\Big(h_1(s)e^{2\mu_2 m_3(s)}\Big)^{1-v}\DD ds
		\\
		\leq &\, 
		\Big(\int\limits_t^\infty h_1(s)e^{2\mu_1 m_3(s)}\DD ds\Big)^v\cdot
		\Big(\int\limits_t^\infty h_1(s)e^{2\mu_2 m_3(s)}\DD ds\Big)^{1-v}
		,
	\end{align*}
	and similarly 
	\[
		\int\limits_0^t h_1(s)e^{-2\lambda m_3(s)}\DD ds\leq 
		\Big(\int\limits_0^t h_2(s)e^{-2\mu_1 m_3(s)}\DD ds\Big)^v\cdot
		\Big(\int\limits_0^t h_2(s)e^{-2\mu_2 m_3(s)}\DD ds\Big)^{1-v}
		.
	\]
	Multiplying these inequalites, and using the elementary inequality 
	\[
		x^vy^{1-v}\leq\max\big\{x,y,x^{\frac 12},y^{\frac 12},xy\big\}
		,\quad x,y>0,v\in(0,1)
		,
	\]
	which is seen by distinguishing cases whether $x,y$ are $\leq 1$ or $\geq 1$ and whether $v\leq\frac 12$ or $v\geq\frac 12$, 
	we obtain \eqref{Q148}.
\item
	Assume now that \Enumref{1} holds. Let $0<\varepsilon<1$, and choose $\mu_1\in A_K^+\cap(-\infty,0)$ and 
	$\mu_2\in A_K^+\cap(0,\infty)$ such that $K\mu_j^{-2}\leq\varepsilon^2$. Then \eqref{Q148} gives 
	$\limsup\limits_{t\to\infty} F(t,0)\leq\varepsilon$. 
	We conclude that 
	\begin{equation}\label{Q143}
		\lim_{t\to\infty}\Big(\int_t^\infty h_1(s)\DD ds\cdot\int_0^t h_2(s)\DD ds\Big)=0
		.
	\end{equation}
	in particular, $\int_0^\infty h_1(s)\DD ds<\infty$. \thref{Q102} implies that $\sigma(\OpA{H})$ is discrete. 
\end{Steps}
\end{proof}

\begin{proof}[Proof of ``\Enumref{2}$\Rightarrow$\Enumref{3}'']
\hfill
\begin{Steps}
\item 
	By \thref{Q102} the condition \eqref{Q143} holds. 
	The condition stated in \Enumref{3} is an extension of this relation to nonzero values of $\lambda$; obviously 
	it holds if and only if 
	\begin{equation}\label{Q144}
		\lim_{t\to\infty}\Big(
		\int_t^\infty h_1(s)e^{2\lambda m_3(s)}\DD ds\cdot\int_0^t h_2(s)e^{-2\lambda m_3(s)}\DD ds
		\Big)=0
		,\ \lambda\in\bb R\setminus\{0\}
		.
	\end{equation}
	Since we have the normalisation $\tr H(t)=1$, it holds that 
	\[
		\text{\eqref{Q143}}
		\ \Leftrightarrow\ 
		\lim_{t\to\infty} t\int_t^\infty h_1(s)\DD ds=0
		\ \Leftrightarrow\ 
		\lim_{t\to\infty} \underbrace{\sup_{x\geq t}\Big(x\int\limits_x^\infty h_1(s)\DD ds\Big)^{\frac 12}}_{\ED p(t)} =0
	\]
	Note that $p$ is nonincreasing. Moreover, again from trace normalisation, 
	\begin{equation}\label{Q145}
		\max\{h_1(t),h_2(t)\}\leq 1,\quad h_3(t)^2\leq h_1(t)h_2(t)\leq\min\{h_1(t),h_2(t)\}
		.
	\end{equation}
\item 
	In this step we show that 
	\begin{equation}\label{Q146}
		|m_3(y)-m_3(x)|\leq p(x)\Big(1+\log\frac yx\Big),\quad 0<x<y
		.
	\end{equation}
	To this end estimate 
	\begin{align*}
		|m_3(y)- & m_3(x)|=\Big|\int_x^y h_3(s)\DD ds\Big|=\Big|\int_x^y \frac 1{\sqrt s}\cdot\sqrt sh_3(s)\DD ds\Big|
		\\
		\leq &\, \Big(\int_x^y\frac 1s\DD ds\Big)^{\frac 12}\Big(\int_x^y sh_3(s)^2\DD ds\Big)^{\frac 12}=
		\big(\log\frac yx\big)^{\frac 12}\Big(\int_x^y sh_3(s)^2\DD ds\Big)^{\frac 12}
		.
	\end{align*}
	Integrating by parts and using \eqref{Q145} gives 
	\begin{align*}
		\int_x^y sh_3(s)^2\DD ds= &\, x\int_x^y h_3(s)^2\DD ds+\int_x^y\Big(\int_x^y h_3(t)^2\DD dt\Big)\DD ds
		\\
		\leq &\, x\int_x^y h_1(s)\DD ds+\int_x^y\Big(\int_x^y h_1(t)\DD dt\Big)\DD ds
		\\
		\leq &\, p(x)^2+\int_x^y \frac 1s\cdot p(s)^2\DD ds\leq p(x)^2\Big(1+\log\frac yx\Big)
		,
	\end{align*}
	and together \eqref{Q146} follows. 
\item 
	Under the assumption that $t_0>0$ is such that $p(t_0)<\frac 1{2|\lambda|}$, we show that
	for all $t\geq t_0$
	\begin{equation}\label{Q147}
		\int_t^\infty h_1(s)e^{2\lambda m_3(s)}\DD ds\cdot\int_0^t h_2(s)e^{-2\lambda m_3(s)}\DD ds
		\leq p(t)^2\cdot\frac{e^2}{(1-2|\lambda|p(t_0))^2}
		.
	\end{equation}
	The product in \eqref{Q147} equals 
	\[
		\int_t^\infty h_1(s)e^{2\lambda(m_3(s)-m_3(t))}\DD ds\cdot\int_0^t h_2(s)e^{-2\lambda(m_3(s)-m_3(t))}\DD ds
		.
	\]
	For the first integral we have 
	\begin{align*}
		\int_t^\infty h_1(s) &\, e^{2\lambda(m_3(s)-m_3(t))}\DD ds \leq
		\int_t^\infty h_1(s)\exp\Big(2|\lambda|p(t)\Big(1+\log\frac st\Big)\Big)\DD ds
		\\
		= &\, \frac e{t^{2|\lambda|p(t)}}\int_t^\infty h_1(s)s^{2|\lambda|p(t)}\DD ds
		=\frac e{t^{-2|\lambda|p(t)}}\bigg[t^{2|\lambda|p(t)}\underbrace{\int_t^\infty h_1(s)\DD ds}_{\leq\frac 1tp(t)^2}
		\\
		&\, +\int_t^\infty 2|\lambda|p(t)s^{2|\lambda|p(t)-1}\cdot
		\Big(\underbrace{\int_s^\infty h_1(x)\DD dx}_{\leq\frac 1sp(s)^2\leq\frac 1sp(t)^2}\Big)\DD ds\bigg]
		\\
		\leq &\, \frac etp(t)^2+\frac e{t^{-2|\lambda|p(t)}}2|\lambda|p(t)^3\frac{t^{2|\lambda|p(t)-1}}{1-2|\lambda|p(t)}
		\leq\frac{p(t)^2}t\cdot\frac{e}{1-2|\lambda|p(t_0)}
		.
	\end{align*}
	For the second integral, 
	\begin{align*}
		\int_0^t h_2(s) &\, e^{-2\lambda(m_3(s)-m_3(t))}\DD ds \leq
		\int_0^t \underbrace{h_2(s)}_{\leq 1}
		\exp\Big(2|\lambda|\mkern-6mu\underbrace{p(t)}_{\leq p(t_0)}\mkern-7mu\Big(1+\log\frac ts\Big)\Big)\DD ds
		\\
		\leq &\, et^{2|\lambda|p(t_0)}\frac{t^{1-2|\lambda|p(t_0)}}{1-2|\lambda|p(t_0)}
		=t\cdot\frac e{1-2|\lambda|p(t_0)}
		.
	\end{align*}
	Putting together, \eqref{Q147} follows. 
\item
	Assume now that \eqref{Q143} holds. Then we find for every $\lambda\in\bb R\setminus\{0\}$ a suitable point $t_0$, 
	and \eqref{Q147} implies \eqref{Q144}.
\end{Steps}
\end{proof}


\printbibliography

{\footnotesize
\begin{flushleft}
	R.\,Romanov\\
	Department of Mathematical Physics\\
	Faculty of Physics, St Petersburg State University\\
	7/9 Universitetskaya nab.\\
	199034 St.Petersburg\\
	RUSSIA\\
	email: morovom@gmail.com\\[3mm]
\end{flushleft}
}
{\footnotesize
\begin{flushleft}
	H.\,Woracek\\
	Institute for Analysis and Scientific Computing\\
	Vienna University of Technology\\
	Wiedner Hauptstra{\ss}e\ 8--10/101\\
	1040 Wien\\
	AUSTRIA\\
	email: harald.woracek@tuwien.ac.at\\[5mm]
\end{flushleft}
}


%
%

%
%
%
\end{document}